\documentclass{amsart}%
\usepackage{amsfonts, amsmath, amssymb}%
\usepackage{graphicx, epic, epsf, pstricks, enumerate}
\usepackage{stmaryrd}

\input xy
\xyoption{all}

\newtheorem{theorem}{Theorem}
\theoremstyle{definition}

\newtheorem{proposition}{Proposition}

\numberwithin{equation}{section}

\newcommand{\db}[1]{\llbracket #1\rrbracket}
 
\newcommand{\Z}{\mathbb{Z}}
\newcommand{\R}{\mathbb{R}}

\begin{document}

\title{On a state model for the $SO(2n)$ Kauffman polynomial}

\author{Carmen Caprau}
\address{Department of Mathematics, California State University, Fresno, CA 93740, USA}
\email{ccaprau@csufresno.edu}
\urladdr{}
\author{David Heywood}
\address{Department of Mathematics, California State University, Fresno, CA 93740, USA}
\email{davaudoo@gmail.com}
\author{Dionne Ibarra}
\address{Department of Mathematics, California State University, Fresno, CA 93740, USA}
\email{luxchasehidknd@yahoo.com}

\date{}

\begin{abstract}

Fran\c{c}ois Jaeger presented the two-variable Kauffman polynomial of an unoriented link $L$ as a weighted sum of HOMFLY-PT polynomials of oriented links associated with $L$. Murakami, Ohtsuki and Yamada (MOY) used planar graphs and a recursive evaluation of these graphs to construct a state model for the $sl(n)$-link invariant (a one-variable specialization of the HOMFLY-PT polynomial). We apply the MOY framework to Jaeger's work, and construct a state summation model for the $SO(2n)$ Kauffman polynomial.
\end{abstract}
\maketitle

\section{Introduction}\label{sec:introd}

The $SO(2n)$ Kauffman polynomial $\db{L}$ of an unoriented link $L$ is a Laurent polynomial in $q$,  uniquely determined by the following axioms: 
\begin{enumerate}
\item[1.] $\db{L_1} = \db{L_2}$, whenever $L_1$ and $L_2$ are regular isotopic links.
 \vspace{0.2cm}

\item[2.] $\left \llbracket \,\raisebox{-8pt}{\includegraphics[height=0.3in]{crossing}}\, \right \rrbracket - 
 \left\llbracket \,\raisebox{-8pt}{\includegraphics[height=0.3in, angle = 90]{crossing}} \,  \right \rrbracket = (q-q^{-1}) \left(  \left \llbracket \,\raisebox{-8pt}{\includegraphics[height=0.3in]{split2}}\, \right \rrbracket - \left \llbracket  \,\raisebox{-8pt}{\includegraphics[height=0.3in]{split1}}\, \right \rrbracket \right)$.
 \vspace{0.2cm}
 \item[3.] $\left \llbracket \,\raisebox{-8pt}{\includegraphics[height=0.3in]{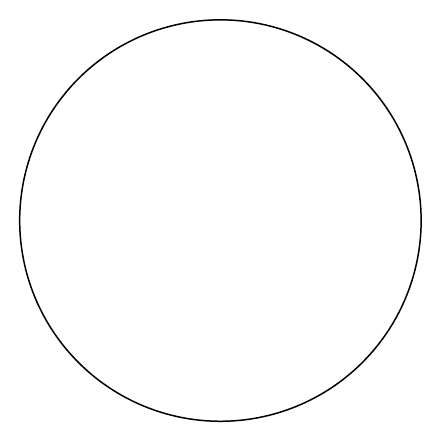}}\, \right \rrbracket = \displaystyle \frac{q^{2n-1} - q^{1-2n}}{q-q^{-1}}+1$.
 \vspace{0.2cm}
\item[4.] $\left \llbracket \reflectbox{\raisebox{-8pt}{\includegraphics[height=0.3in,width = 0.3in, angle = 90]{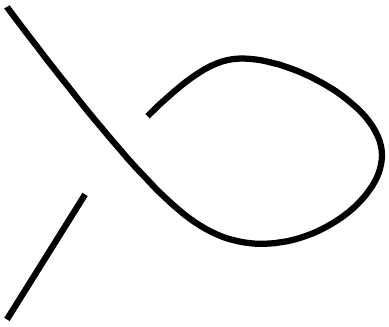}}}\right \rrbracket  = q^{2n-1} \left \llbracket \, \raisebox{-8pt}{\includegraphics[height=0.25in, width = 0.3in]{arc2}}\,\,\right \rrbracket , \quad \left \llbracket\raisebox{-8pt}{\includegraphics[height=0.3in, width = 0.3in, angle = 90]{negkink}}\right \rrbracket  = q^{1-2n} \left \llbracket\,\raisebox{-8pt}{\includegraphics[height=0.25in, width = 0.3in]{arc2}}\,\,\right \rrbracket$.
\end{enumerate}

The diagrams in both sides of the second or fourth equations represent parts of larger link diagrams that are identical except near a point where they look as indicated. For more details about this polynomial (and its two-variable extension, namely the Dubrovnik version of the two-variable Kauffman polynomial) we refer the reader to~\cite{K1, K2}. 

Kauffman and Vogel~\cite{KV} extended the two-variable Dubrovnik polynomial to a three-variable rational function for knotted 4-valent graphs (4-valent graphs embedded in $\mathbb{R}^3$) with rigid vertices. For the case of the $SO(2n)$ Kauffman polynomial, this extension is obtained by defining 
\begin{eqnarray*}
 \left \llbracket \,\raisebox{-8pt}{\includegraphics[height=0.3in]{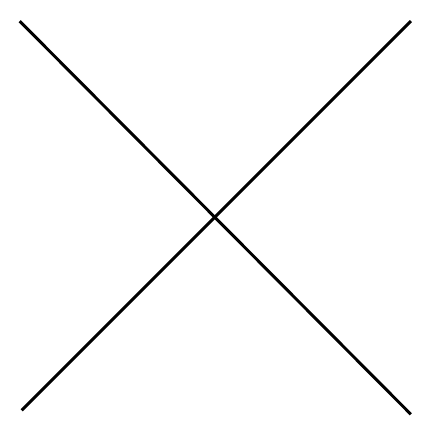}} \,\right \rrbracket : &=& \left \llbracket \,\raisebox{-8pt}{\includegraphics[height=0.3in, angle = 90]{crossing}}\, \right \rrbracket - q \left \llbracket \,\raisebox{-8pt}{\includegraphics[height=0.3in]{split1}}\,\right \rrbracket -q^{-1}  \left \llbracket\,\raisebox{-8pt}{\includegraphics[height=0.3in]{split2}}\,\right \rrbracket \\
 &=&   \left \llbracket\,\raisebox{-8pt}{\includegraphics[height=0.3in]{crossing}}\,\right \rrbracket - q \left \llbracket\,\raisebox{-8pt}{\includegraphics[height=0.3in]{split2}}\, \right \rrbracket -q^{-1}  \left \llbracket\,\raisebox{-8pt}{\includegraphics[height=0.3in]{split1}}\, \right \rrbracket
 \end{eqnarray*}
 That is, the invariant for knotted 4-valent graphs with rigid vertices is defined in terms of the $SO(2n)$ Kauffman polynomial. In~\cite{KV}, it was also shown that the resulting polynomial of a knotted 4-valent graph satisfies certain graphical relations, which determine values for each unoriented planar 4-valent graph by recursive formulas defined entirely in the category of planar graphs.
 
The results in ~\cite{KV} imply that there is a state model for the Kauffman polynomial of an unoriented link via planar 4-valent graphs. This model can also be deduced from Carpentier's work~\cite{C} on the Kauffman-Vogel polynomial by changing one's perspective (the focus of Carpentier's paper is on invariants for graphs rather than on the Kauffman polynomial for links). A somewhat similar approach was used in~\cite{CT} to construct a rational function in three variables which is an invariant of regular isotopy of unoriented links, and provides a state summation model for the Dubrovnik version of the two-variable Kauffman polynomial. The corresponding state model makes use of a special type of planar trivalent graphs. 
 
Fran\c{c}ois Jaeger found a relationship between the  two-variable Kauffman polynomial and the regular isotopy version of the HOMFLY-PT polynomial. He showed that the Kauffman polynomial of an unoriented link $L$ can be obtained as a weighted sum of HOMFLY-PT polynomials of oriented links associated with $L$. For a brief description of Jaeger's construction we refer the reader to~\cite{K2}.
Murakami, Ohtsuki and Yamada (MOY) used planar trivalent graphs to construct in~\cite{MOY} a beautiful graphical calculus for the $sl(n)$-link polynomial (a one-variable specialization of the HOMFLY-PT polynomial). 

The motivation for this paper has its source in the following, natural, questions: Is there a way to apply the MOY model to Jaeger's formula and derive a state summation model for the $SO(2n)$ Kauffman polynomial? And if so, how is the resulting state model for the $SO(2n)$ Kauffman polynomial related to the one implicitly given in~\cite{KV}?

We slightly alter the MOY model for the $sl(n)$-link polynomial by working with (planar, cross-like oriented) 4-valent graphs instead of trivalent graphs. By implementing the MOY model into Jaeger's construction, we show that in order to construct a state model for the Kauffman polynomial it is not sufficient to allow only cross-like oriented 4-valent graphs but also alternating oriented vertices. The skein formalism that we obtain is as follows:

\[  \left \llbracket \,\raisebox{-8pt}{\includegraphics[height=.3in]{crossing}}\,\right \rrbracket  = q \left \llbracket\, \raisebox{-8pt}{\includegraphics[height=.3in]{split2}}\,\right \rrbracket
		+ q^{-1} \left \llbracket \raisebox{-8pt}{\includegraphics[height=.3in]{split1}}\right \rrbracket 
	 	-  \left \llbracket \,\raisebox{-8pt}{\includegraphics[height=.3in]{vertnc}}\,\right \rrbracket \]

\[\left \llbracket \raisebox{-8pt}{\includegraphics[height=.3in]{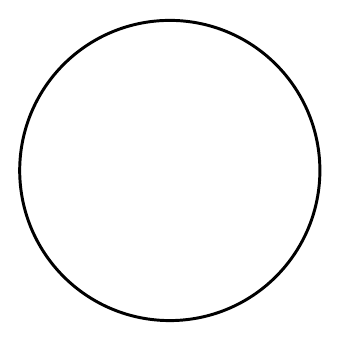}} \right \rrbracket = [2n-1]+1\]
\[
\left \llbracket \raisebox{-10pt}{\includegraphics[height=.35in]{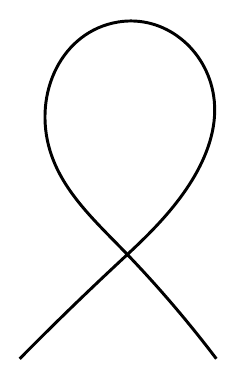}} \,\right \rrbracket = ([2n-2]+[2])\left \llbracket \raisebox{-3pt}{\includegraphics[width=.4in]{arc2}}\, \right \rrbracket  \] 
\[\left \llbracket\, \raisebox{-8pt}{\includegraphics[height=.3in]{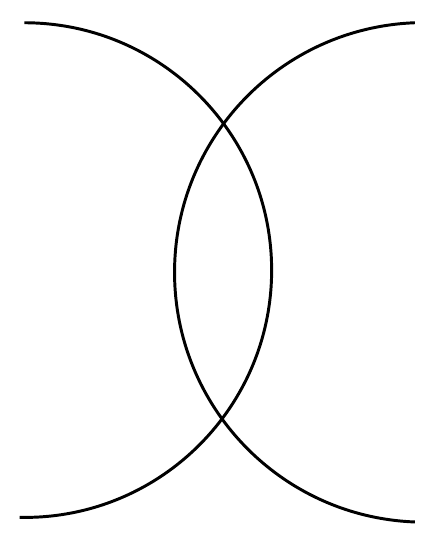}} \,\right \rrbracket = ([2n-3]+1)\left \llbracket\, \raisebox{-8pt}{\includegraphics[height=.3in]{split2}}\, \right \rrbracket + [2]\left \llbracket\, \raisebox{-8pt}{\includegraphics[height=.3in]{vertnc}}\, \right \rrbracket \]
\begin{eqnarray*}
&&\left \llbracket \raisebox{-8pt}{\includegraphics[height=.3in]{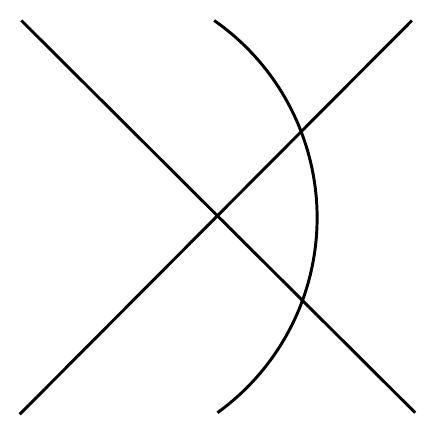}} \right \rrbracket
	+ \left \llbracket \raisebox{-8pt}{\includegraphics[height=.3in]{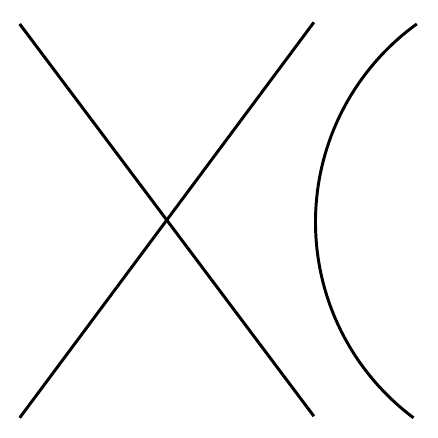}} \right \rrbracket
	- \left \llbracket \raisebox{-8pt}{\includegraphics[height=.3in]{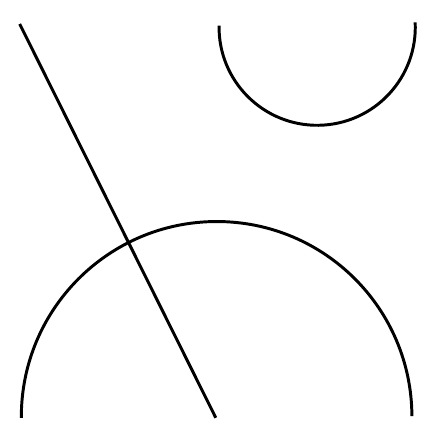}} \right \rrbracket
	- \left \llbracket \raisebox{-8pt}{\includegraphics[height=.3in]{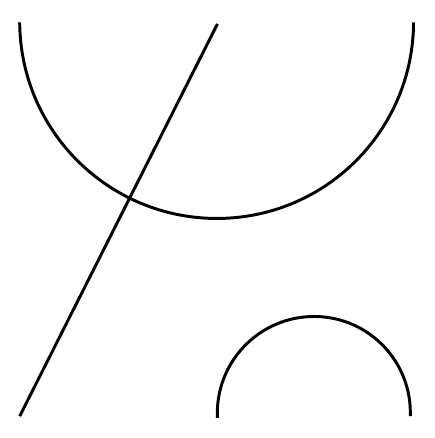}} \right \rrbracket
	- [2n-4] \left \llbracket \raisebox{-8pt}{\includegraphics[height=.3in]{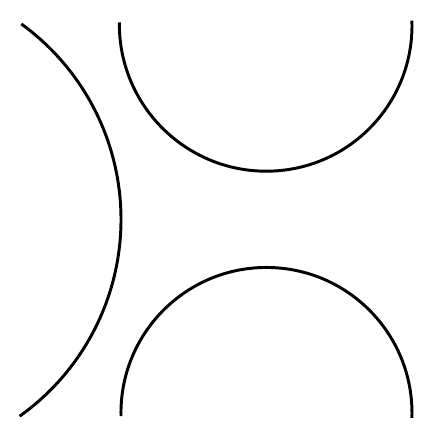}} \right \rrbracket = \\
	&& \left \llbracket \raisebox{-8pt}{\includegraphics[height=.3in]{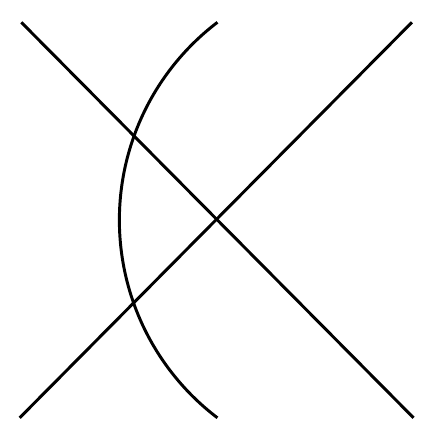}} \right \rrbracket
	+ \left \llbracket \raisebox{-8pt}{\includegraphics[height=.3in]{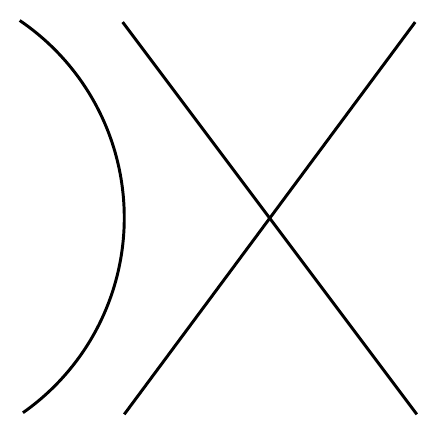}} \right \rrbracket
	- \left \llbracket \raisebox{-8pt}{\includegraphics[height=.3in]{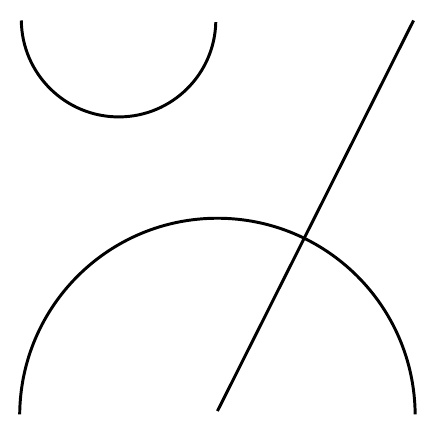}} \right \rrbracket
	- \left \llbracket \raisebox{-8pt}{\includegraphics[height=.3in]{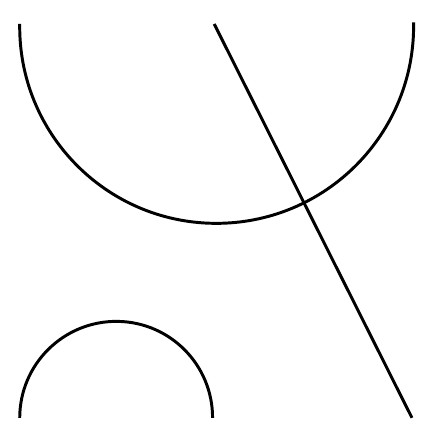}} \right \rrbracket
	- [2n-4] \left \llbracket \raisebox{-8pt}{\includegraphics[height=.3in]{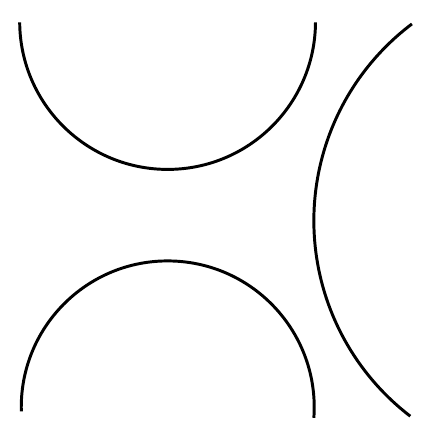}} \right \rrbracket 
	\end{eqnarray*}
	where $\displaystyle [n] = \displaystyle \frac{q^n - q^{-n}}{q - q^{-1}}$, and $n \in \Z$ with $n \geq 2$.
	
Comparing the graph skein relations above with the graphical relations derived by Kauffman and Vogel in~\cite{KV}, it is not hard to see that the state model for the $SO(2n)$ Kauffman polynomial that we arrive to is essentially the same as that implied by the work in~\cite{KV}  (up to a negative sign for the weight  received by the ``flat resolution" of a crossing), and that given in ~\cite[Subsection 5.1]{CT} (up to a change of variables). We would like to point out that Hao Wu~\cite{WU} used a different approach to write the Kauffman-Vogel graph polynomial as a state sum of the MOY graph polynomial.

The paper is organized as follows: In Section~\ref{sec:MOY} we provide a version of the Murakami-Ohtsuki-Yamada state model for the $sl(n)$-link polynomial, and in Section~\ref{sec:Jaeger} we review Jaeger's formula for the Kauffman polynomial. The heart of the paper is Section~\ref{sec:main}, in which we derive the state model for the $SO(2n)$ Kauffman polynomial.
 

\section{The MOY state model for the $sl(n)$ polynomial}\label{sec:MOY}

In this section, we give the Murakami-Ohtsuki-Yamada~\cite{MOY} state model for the regular isotopy version of the $sl(n)$ polynomial of an oriented link $L$. The $sl(n)$ polynomial is a one-variable specialization of the well-known HOMFLY-PT polynomial (see~\cite{HOMFLY, PT}).
Let $D$ be a generic diagram of $L$ containing $c$ crossings. We resolve each crossing of $D$ in the two ways shown below:
\[ \raisebox{-8pt}{\includegraphics[height=0.28in]{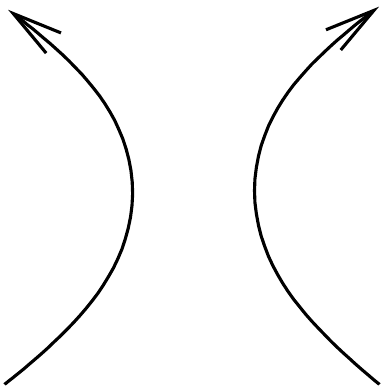}} \longleftarrow \raisebox{-8pt}{\includegraphics[height=0.3in]{crossingpos}} \ , \  \raisebox{-8pt}{\includegraphics[height=0.3in]{crossingneg}} \longrightarrow \raisebox{-8pt}{\includegraphics[height=0.28in]{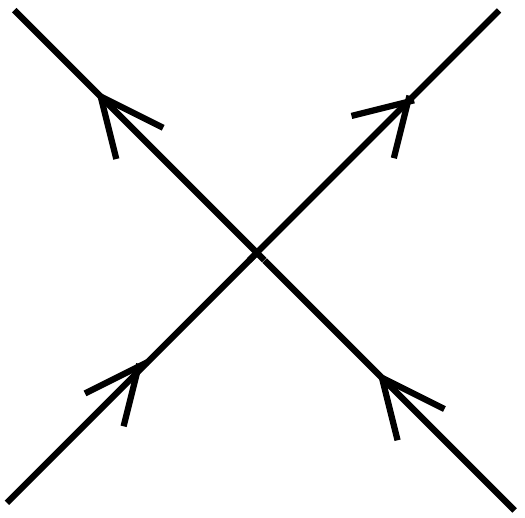}}   \]
This process yields $2^c$ resolutions (states) corresponding to the link diagram $D$. A resolution $\Gamma$ of $D$ is a $4$-valent oriented planar graph in $\R^2$, possibly with loops with no vertices, such that each vertex is crossing-type oriented:  \raisebox{-6pt}{\includegraphics[height = 0.23in]{vertex}}. There is a well-defined Laurent polynomial $R(\Gamma) \in \Z[q, q^{-1}]$ associated to a resolutions $\Gamma$, such that it satisfies the skein relations depicted in Figure~\ref{fig:web skein relations}, where $[n] = \displaystyle \frac{q^n - q^{-n}}{q - q^{-1}}$, and $n \in \Z$ with $n \geq 2$ (the symbol $R$ is omitted in the graph skein relations to avoid clutter). We will refer to $R(\Gamma)$ as the \textit{MOY graph polynomial} (see~\cite{MOY}).

\begin{figure}[ht]
\raisebox{-8pt}{\includegraphics[height = 0.35in]{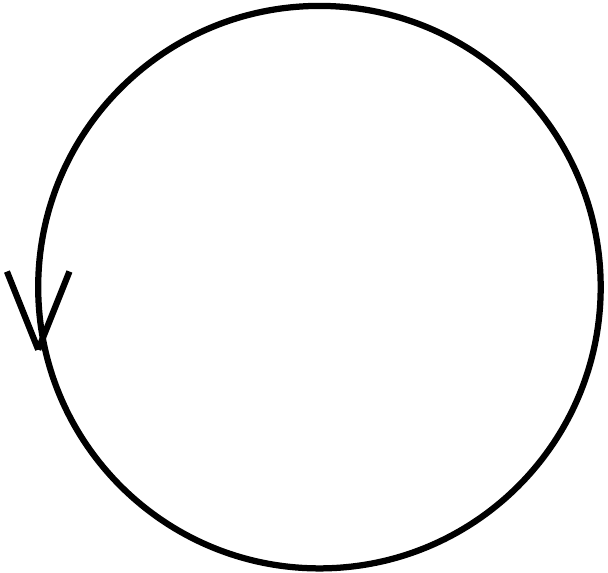}} = [n] \hspace{1cm} \raisebox{-15pt}{\includegraphics[height = 0.45in]{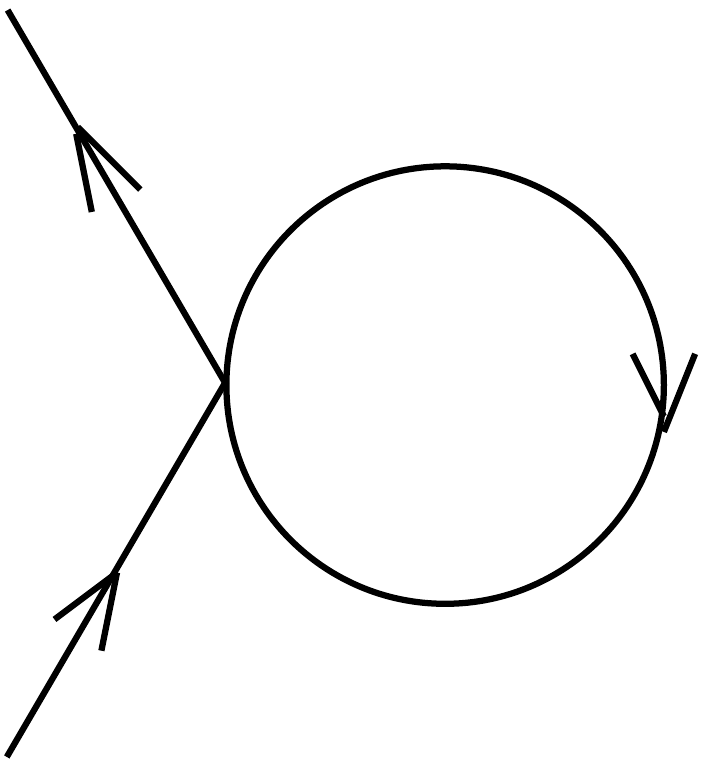}} = [n -1]\raisebox{-15pt}{\includegraphics[height = 0.45in]{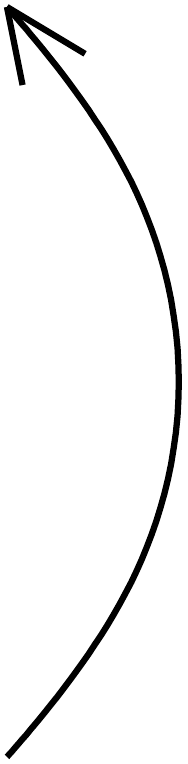}}

\vspace{0.3cm}
\raisebox{15pt}{\includegraphics[height = 0.65in, angle = 270]{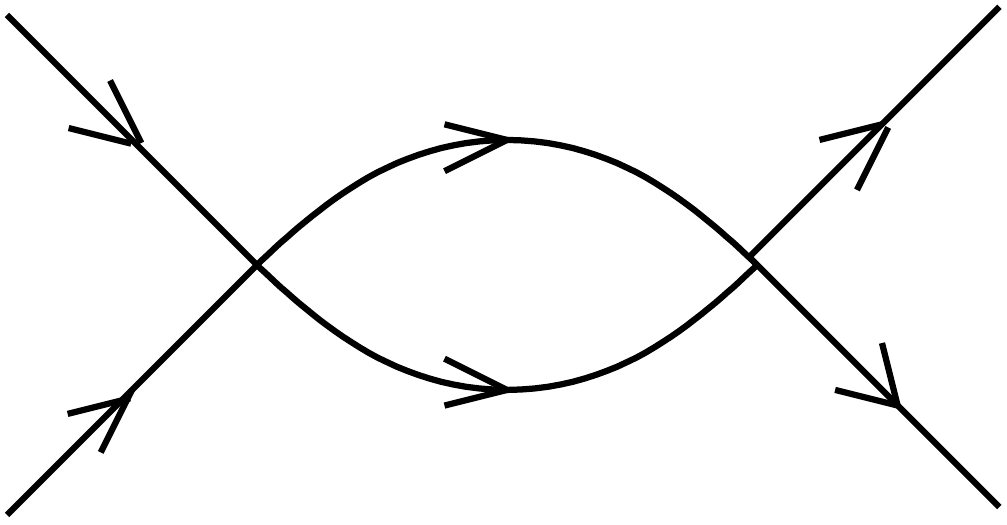}} = [2] \raisebox{15pt}{\includegraphics[height = 0.35in, angle = 270]{vertex}}

\vspace{0.3cm}
\raisebox{15pt}{\includegraphics[height = 0.65in, angle = 270]{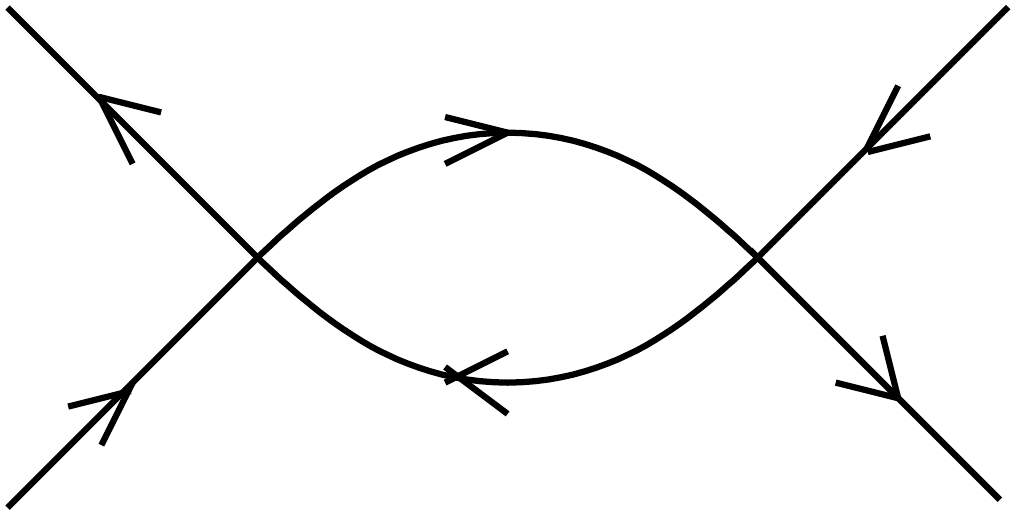}}  =  \raisebox{15pt}{\includegraphics[height = 0.45in, angle = 270]{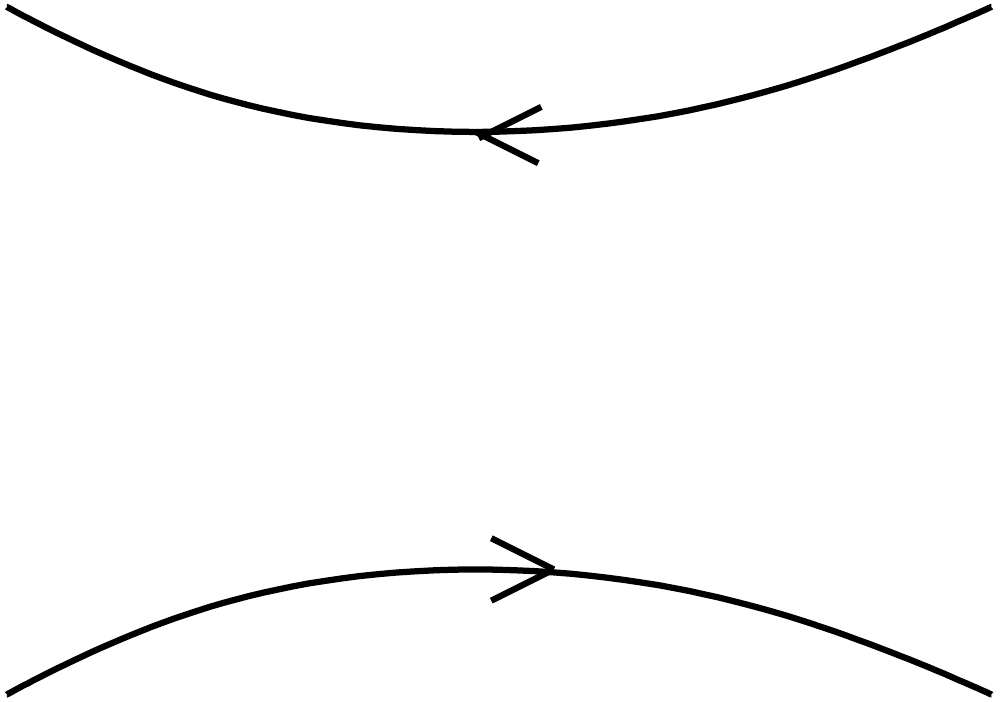}} + [n - 2] \raisebox{15pt}{\includegraphics[height = 0.35in, angle = 270]{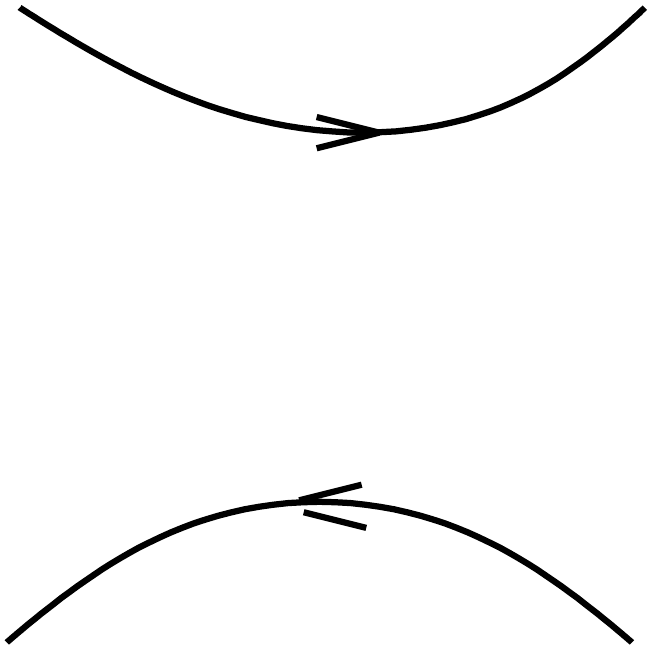}}

\vspace{0.3cm}
\raisebox{-20pt}{\includegraphics[height = 0.65in]{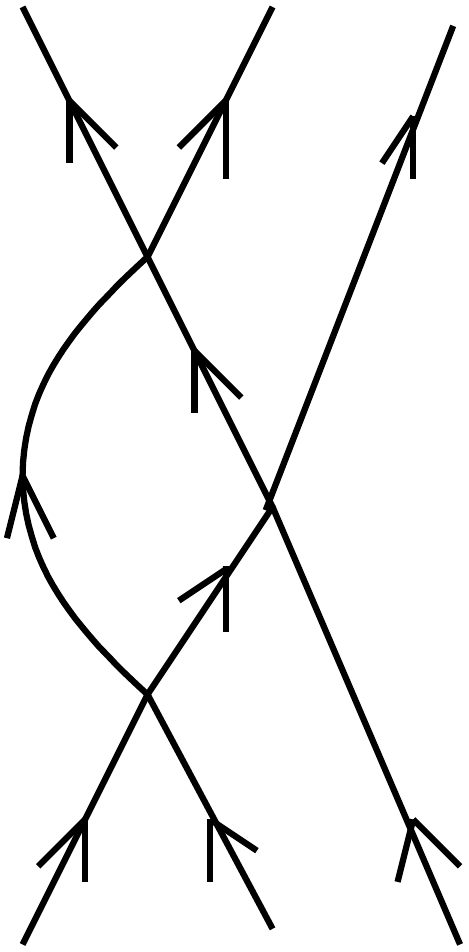}} \ + \  \raisebox{-20pt}{\includegraphics[height = 0.65in]{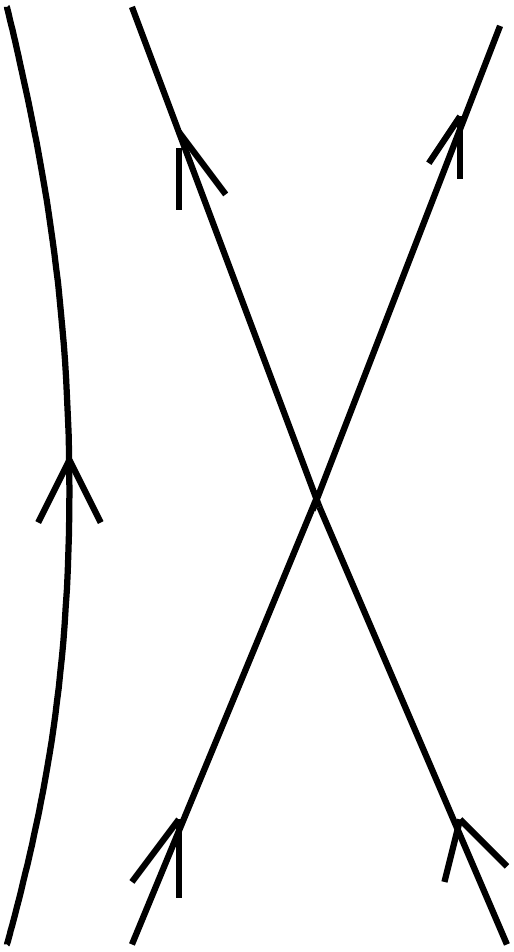}}  = \reflectbox{\raisebox{-20pt}{\includegraphics[height = 0.65in]{skein-left1}}} \  + \  \reflectbox{\raisebox{-20pt}{\includegraphics[height = 0.65in]{skein-left2}}} 

\vspace{0.3cm}
\raisebox{-20pt}{\includegraphics[height = 0.65in]{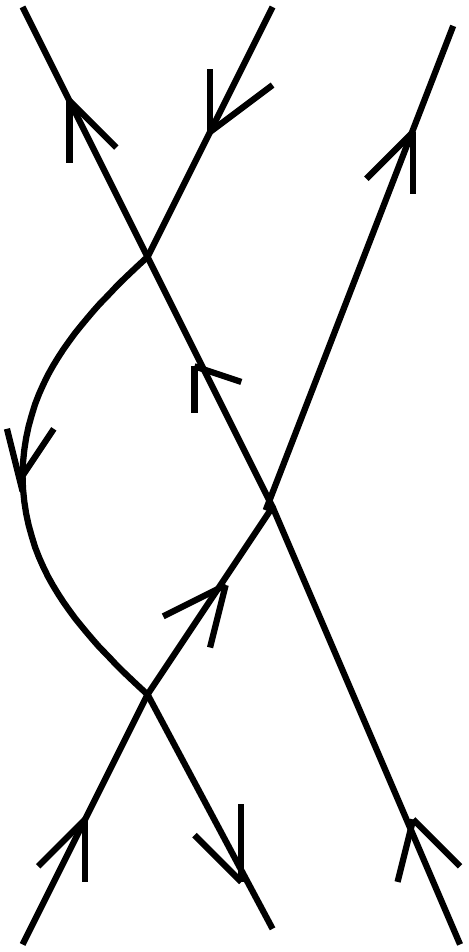}} \ + [n - 3] \  \reflectbox{\raisebox{-20pt}{\includegraphics[height = 0.65in]{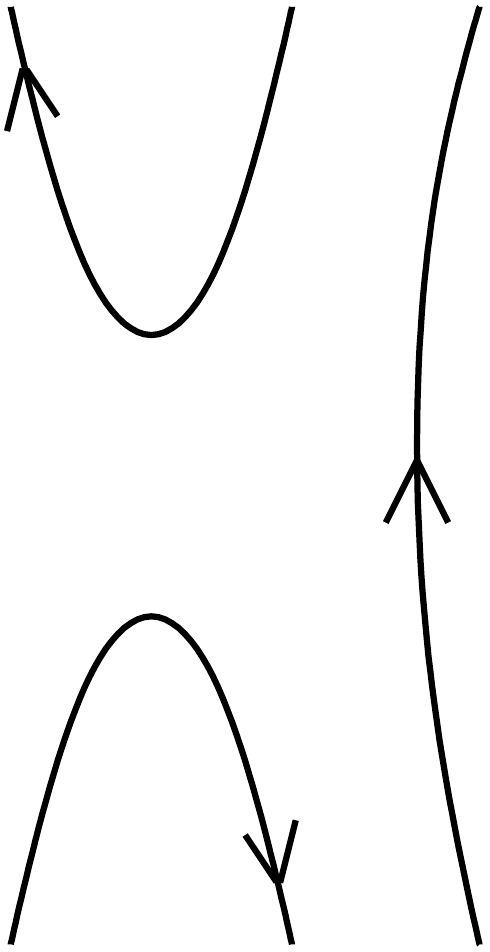}}}  = \reflectbox{\raisebox{-20pt}{\includegraphics[height = 0.65in]{skein-left3}}} \  + [n - 3] \  \raisebox{-20pt}{\includegraphics[height = 0.65in]{skein-left4}} 

\caption{Web skein relations}\label{fig:web skein relations}
\end{figure}

Decompose each crossing in $D$ as explained in Figure~\ref{fig:decomposition}, and form the following linear combination of the MOY evaluations of all $2^c$ resolutions $\Gamma$ of $D$:

\[ R(D) = \sum_{\Gamma} a_\Gamma R(\Gamma), \]
where the coefficients $a_\Gamma \in \Z[q, q^{-1}]$ are given by the rules depicted in Figure~\ref{fig:decomposition}.
\begin{figure}[ht]
\begin{eqnarray*}
R\left( \, \raisebox{-8pt}{\includegraphics[height=0.3in]{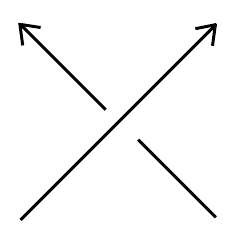}}\, \right ) &=& qR \left (\, \raisebox{-8pt}{\includegraphics[height=0.28in]{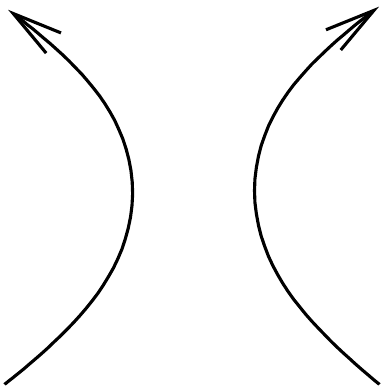}} \, \right ) - R \left ( \, \raisebox{-8pt}{\includegraphics[height=0.28in]{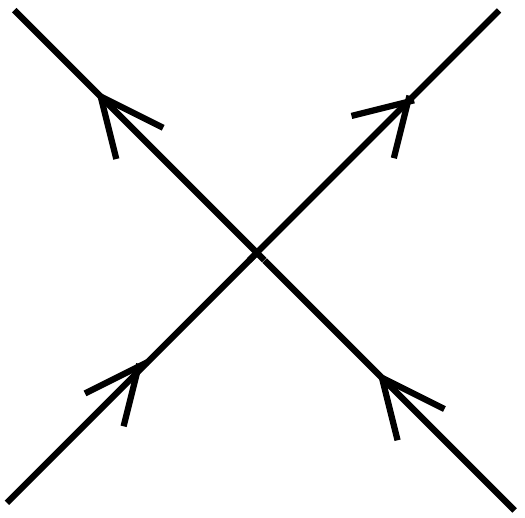}} \, \right ) \\
R \left( \,\raisebox{-8pt}{\includegraphics[height=0.3in]{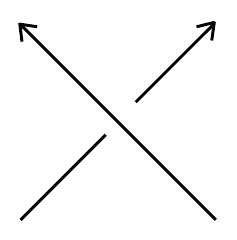}}\, \right ) &=& q^{-1}R \left ( \, \raisebox{-8pt}{\includegraphics[height=0.28in]{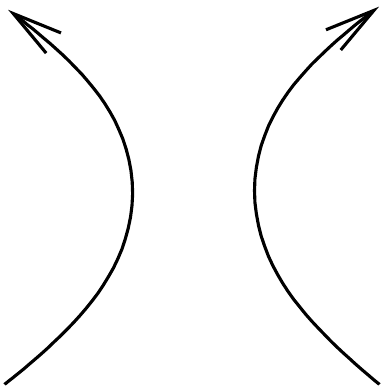}}\, \right ) - R \left ( \, \raisebox{-8pt}{\includegraphics[height=0.28in]{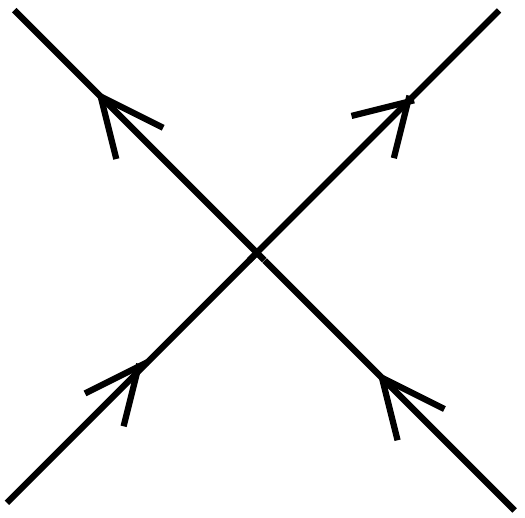}} \, \right )
\end{eqnarray*}
\caption{Decomposition of crossings}\label{fig:decomposition}
\end{figure}

It is an enjoyable exercise to verify that $R(D_1) = R(D_2)$, whenever diagrams $D_1$ and $D_2$ differ by a Reidemeister II or III move.
Excluding rightmost terms from the decomposition rules of crossings, we obtain Conway's skein relation: 
\[ R\left(\, \raisebox{-8pt}{\includegraphics[height=0.3in]{crossingneg}}\, \right) - R\left(\, \raisebox{-8pt}{\includegraphics[height=0.3in]{crossingpos}} \,\right ) = (q - q^{-1})R\left(\, \raisebox{-8pt}{\includegraphics[height=0.28in]{orienres}}\, \right). \]
We note that $R(L): = R(D)$ is the regular isotopy version of the $sl(n)$ polynomial of the link $L$, and that it satisfies the following:
\[ R\left(\, \raisebox{-8pt}{\includegraphics[height=0.3in]{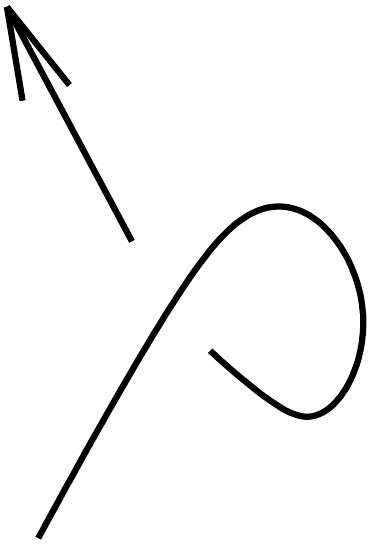}}\, \right)  = q^n R\left ( \ \raisebox{-8pt}{\includegraphics[height = 0.3in]{arc}} \ \right ) \ \  \text{and} \ \  R\left(\, \reflectbox{\raisebox{-8pt}{\includegraphics[height=0.3in]{poskink}}}\, \right)  = q^{-n} R\left ( \ \reflectbox{\raisebox{-8pt}{\includegraphics[height = 0.3in]{arc}}} \ \right ).  \]

\section{Jaeger's model for the Kauffman polynomial}\label{sec:Jaeger}

In the late 80's, Fran\c{c}ois Jaeger found a relationship between the  two-variable Kauffman polynomial and the regular isotopy version of the HOMFLY-PT polynomial. He showed that the Kauffman polynomial of an unoriented link $L$ can be obtained as a weighted sum of HOMFLY-PT polynomials of oriented links associated with $L$. Since this construction is only briefly described in~\cite{K2}, we provide here a thorough exposition of it, which is necessary in order to understand our main Section~\ref{sec:main}. Moreover, we describe Jaeger's model for the $SO(2n)$ Kauffman polynomial by considering the sl(n)-link invariant instead of the HOMFLY-PT polynomial. 

Given an unoriented link diagram $L$, splice some of the crossings of $L$ and orient the resulting link. This results in a \textit{state} for the expansion $\db{L}$. Each state receives a certain weight, according to the following skein relation:

\[
 \left\llbracket \,\raisebox{-8pt}{\includegraphics[scale=0.18]{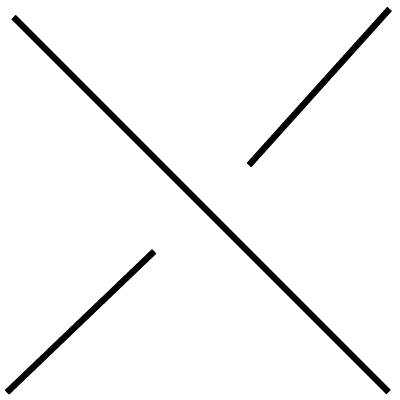}}\, \right\rrbracket = (q -q^{-1}) \left( \left[\raisebox{-6 pt}{\includegraphics[scale=0.3]{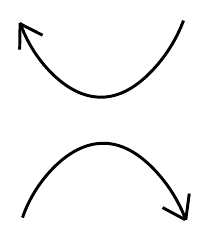}} \right]
- \left[\raisebox{-4 pt}{\includegraphics[scale=0.3]{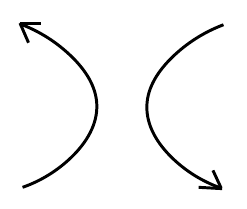}} \right] \right) 
+ \left[ \raisebox{14 pt}{\includegraphics[angle=270, scale=0.3]{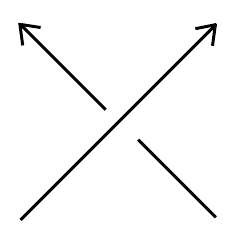}} \right] +\left[ \raisebox{14 pt}{\includegraphics[angle=180, scale=0.3]{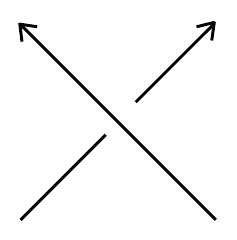}} \right] 
+\left[ \raisebox{-6 pt}{\includegraphics[angle=90, scale=0.3]{crossingpos.pdf}} \right] 
+\left[ \raisebox{-6 pt}{\includegraphics[scale=0.3]{crossingneg.pdf}} \right] \quad (*)
\]

It is important to remark that the formula $(*)$ requires states that are oriented in a globally compatible way as oriented link diagrams. Moreover, observe that the orientation and the weight of a state are determined by how the crossings are spliced. When approaching a crossing by traveling along the understrand, a splicing is obtained by either turning right or left at that crossing. In both cases, the strands of the splicing are oriented according to the direction of the traveling. If the crossing is spliced by turning right, then it receives the weight $q -q^{-1}$, and if it is spliced by turning left, it receives the weight $-(q -q^{-1})$. If a crossing is left unspliced, its weight (in the total weight of the state) is equal to $1$. 

The \textit{weight} $b_{\sigma}$ of a state $\sigma$ is obtained by taking the product of the weights $\pm (q-q^{-1})$ or $1$ according to the skein relation ($*$). Define the \textit{evaluation} of a state $\sigma$ by the formula
\[ [ \sigma ] = (q^{1-n}) ^{\textrm{rot}(\sigma)} R(\sigma),\]
where $\textrm{rot}(\sigma)$ is the \textit{rotation number} of the oriented link diagram $\sigma$, and $R(\sigma)$ is the regular isotopy version of the $sl(n)$ polynomial of $\sigma$.

The rotation number (also called the Whitney degree) of an oriented link diagram is obtained by splicing every crossing according to its orientation, and then adding the rotation numbers of all of the resulting Seifert circles, where a counterclockwise oriented circle contributes a $+1$, and a clockwise oriented circle contributes a $-1$. It is well-known that the rotation number is a regular isotopy invariant for oriented links.

Equipped with the above definitions and conventions, we are ready to state Jaeger's theorem.

\begin{theorem} (Jaeger) The Kauffman polynomial $\db{L}$ of an unoriented link diagram $L$ can be obtained as follows:
\[ \db{L} = \sum_{\sigma} b_\sigma [ \sigma ],\]
where the sum is over all states $\sigma$ associated with $L$ that have globally compatible orientations.
\end{theorem}

\begin{proof}
First note that the Conway identity holds for $[\, \cdot \,]$:
\begin{eqnarray*}
\left[ \raisebox{14 pt}{\includegraphics[angle=270, scale=0.3]{crossingpos.pdf}} \right] - \left[ \raisebox{14 pt}{\includegraphics[angle=270, scale=0.3]{crossingneg.pdf}} \right] &=& q^{(1-n) \textrm{rot}\left(\raisebox{10pt}{\includegraphics[angle=270, scale=0.2]{crossingpos.pdf}}\right)} R\left(\raisebox{14 pt}{\includegraphics[angle=270, scale=0.3]{crossingpos.pdf}}\right) - q^{(1-n)\textrm{rot}\left(\raisebox{10pt}{\includegraphics[angle=270, scale=0.2]{crossingneg.pdf}}\right)} R\left(\raisebox{14 pt}{\includegraphics[angle=270, scale=0.3]{crossingneg.pdf}}\right) \\
&=& q^{(1-n) \textrm{rot}\left(\raisebox{-5pt}{\includegraphics[scale=0.2]{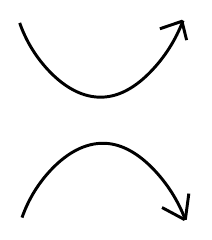}}\right)} \left ( R\left(\raisebox{14 pt}{\includegraphics[angle=270, scale=0.3]{crossingpos.pdf}}\right) - R\left(\raisebox{14 pt}{\includegraphics[angle=270, scale=0.3]{crossingneg.pdf}}\right) \right)\\
&=& (q-q^{-1}) q^{(1-n) \textrm{rot}\left(\raisebox{-5pt}{\includegraphics[scale=0.2]{splita.pdf}}\right)} R\left(\raisebox{-7 pt}{\includegraphics[scale=0.3]{splita.pdf}}\right) \\
&=& (q-q^{-1})\left[\raisebox{-7 pt}{\includegraphics[scale=0.3]{splita.pdf}}\right].
\end{eqnarray*}
Then,
\begin{eqnarray*}
 \left\llbracket \raisebox{-8pt}{\includegraphics[scale=0.18]{crossing.pdf}} \right\rrbracket -
\left\llbracket \raisebox{-8pt}{\includegraphics[angle=90, scale=0.18]{crossing.pdf}} \right\rrbracket
&=& (q-q^{-1}) \left( \left[ \raisebox{-7 pt}{\includegraphics[scale=0.3]{split2or.pdf}} \right]-
\left[ \raisebox{-5 pt}{\includegraphics[scale=0.3]{split1or.pdf}} \right]\right) 
+\left[ \raisebox{14 pt}{\includegraphics[angle=270, scale=0.3]{crossingpos.pdf}} \right]
+\left[ \raisebox{14 pt}{\includegraphics[angle=180, scale=0.3]{crossingneg.pdf}} \right]
+\left[ \raisebox{-6 pt}{\includegraphics[angle=90, scale=0.3]{crossingpos.pdf}} \right]
+\left[ \raisebox{-6 pt}{\includegraphics[scale=0.3]{crossingneg.pdf}} \right] \\
&-&  (q-q^{-1})  \left( \left[ \raisebox{-5 pt}{\includegraphics[angle=90, scale=0.3]{split2or.pdf}} \right]-
\left[ \raisebox{-7 pt}{\includegraphics[angle=90, scale=0.3]{split1or.pdf}} \right] \right) 
-\left[ \raisebox{14 pt}{\includegraphics[angle=270, scale=0.3]{crossingneg.pdf}} \right]
-\left[ \raisebox{14 pt}{\includegraphics[angle=180, scale=0.3]{crossingpos.pdf}} \right]
-\left[ \raisebox{-6 pt}{\includegraphics[angle=90, scale=0.3]{crossingneg.pdf}} \right]
-\left[ \raisebox{-6 pt}{\includegraphics[scale=0.3]{crossingpos.pdf}} \right], 
\end{eqnarray*}
and by the Conway identity, we obtain
\begin{eqnarray*}
\left\llbracket \raisebox{-8pt}{\includegraphics[scale=0.18]{crossing.pdf}} \right\rrbracket -
\left\llbracket \raisebox{-8pt}{\includegraphics[angle=90, scale=0.18]{crossing.pdf}} \right\rrbracket 
&=& (q-q^{-1}) \left( \left[ \raisebox{-7 pt}{\includegraphics[scale=0.3]{split2or.pdf}} \right]
+\left[ \raisebox{14 pt}{\includegraphics[angle=180, scale=0.3]{splita.pdf}} \right] 
+\left[ \raisebox{-7 pt}{\includegraphics[ scale=0.3]{splita.pdf}} \right] 
+\left[ \raisebox{-7 pt}{\includegraphics[angle=90, scale=0.3]{split1or.pdf}} \right]\right)  \\
&-&  (q-q^{-1}) \left( \left[ \raisebox{-5 pt}{\includegraphics[angle=90, scale=0.3]{split2or.pdf}} \right]
+\left[ \raisebox{14 pt}{\includegraphics[angle=270, scale=0.3]{splita.pdf}} \right] 
+\left[ \raisebox{-5 pt}{\includegraphics[scale=0.3]{split1or.pdf}} \right]
+\left[ \raisebox{-7 pt}{\includegraphics[angle=90, scale=0.3]{splita.pdf}} \right] \right) \\
&=& (q-q^{-1}) \left( \left\llbracket \raisebox{-7 pt}{\includegraphics[scale=0.3]{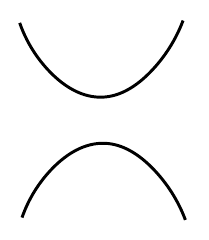}} \right\rrbracket-
\left\llbracket \raisebox{-5 pt}{\includegraphics[scale=0.3]{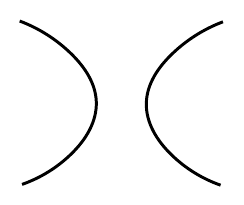}} \right\rrbracket \right).
\end{eqnarray*}
Observe that \[\left[\raisebox{-8pt}{\includegraphics[height=.3in]{circle}} \right] = q^{(1-n)\textrm{rot}  \left( \raisebox{-6pt}{\includegraphics[height=.25in]{circle}}\right) } R\left(\raisebox{-8pt}{\includegraphics[height=.3in]{circle}}\right) = q^{1-n}[n],\] 
\[\left[\raisebox{-8pt}{\reflectbox{\includegraphics[height=.3in]{circle}} } \right] = q^{(1-n)\textrm{rot}  \left( \raisebox{-6pt}{\reflectbox{\includegraphics[height=.25in]{circle}}}\right) } R\left(\raisebox{-8pt}{\reflectbox{\includegraphics[height=.3in]{circle}}}\right) = q^{n-1}[n],\]
and, therefore we have
\[ \left \llbracket \raisebox{-8pt}{\includegraphics[height=.3in]{unknot1}} \right \rrbracket = \left[\raisebox{-8pt}{\includegraphics[height=.3in]{circle}}  \right]  + \left[\raisebox{-8pt}{\reflectbox{\includegraphics[height=.3in]{circle}} } \right] = (q^{1-n} + q^{n-1})[n] =  \frac{q^{2n-1}-q^{1-2n}}{q-q^{-1}}+1.\]
Moreover,
\begin{eqnarray*} 
\left[ \raisebox{-11 pt}{\includegraphics[scale=0.3]{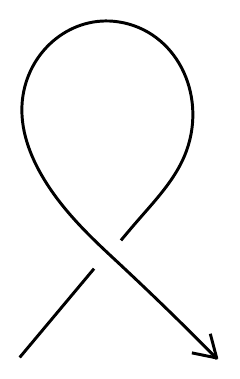}} \right] &=&
q^{(1-n)\text{rot} \left( \raisebox{-7 pt}{\includegraphics[scale=0.2]{kinkpos.pdf}} \right)} 
R\left(\raisebox{-11 pt}{\includegraphics[scale=0.3]{kinkpos.pdf}}\right) 
= q^{(1-n) \left (1+\text{rot} \left( \raisebox{8 pt}{\includegraphics[angle=180, scale=0.35]{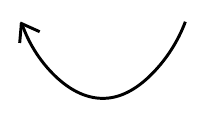}} \right)\right)} q^n R \left( \raisebox{10 pt}{\includegraphics[angle=180, scale=0.5]{line2.pdf}}\right) \\
&=& q \left[ \raisebox{10pt}{\includegraphics[angle=180, scale=0.5]{line2.pdf}} \right].
\end{eqnarray*}

\begin{eqnarray*} 
\left[ \raisebox{-11 pt}{\includegraphics[scale=0.3]{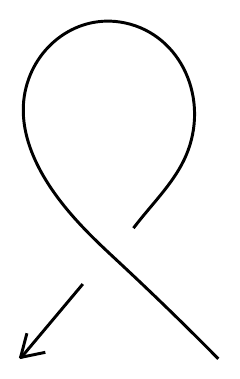}} \right] &=&
q^{(1-n){\text{rot} \left( \raisebox{-7 pt}{\includegraphics[scale=0.2]{kinkpos1.pdf}} \right)} }
R\left(\raisebox{-11 pt}{\includegraphics[scale=0.3]{kinkpos1.pdf}}\right) 
= q^{(1-n)\left(-1 + \text{rot} \left( \raisebox{8 pt}{\includegraphics[angle=180, scale=0.35]{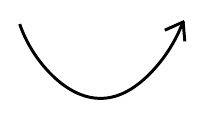}} \right)\right)} q^n R \left(\raisebox{10 pt}{\includegraphics[angle=180, scale=0.5]{line1.pdf}}\right) \\
&=& q^{2n-1}\left[ \raisebox{10pt}{\includegraphics[angle=180, scale=0.5]{line1.pdf}} \right].
\end{eqnarray*}
Therefore,

\begin{eqnarray*}
 \left\llbracket \raisebox{-15 pt}{\includegraphics[scale=0.3]{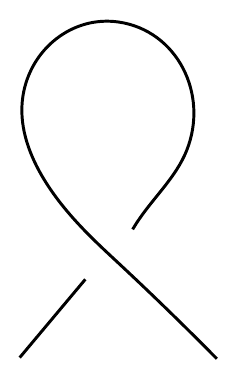}} \right\rrbracket 
& =&  (q-q^{-1}) \left( \left[ \raisebox{20 pt}{\includegraphics[angle = 180, scale=0.3]{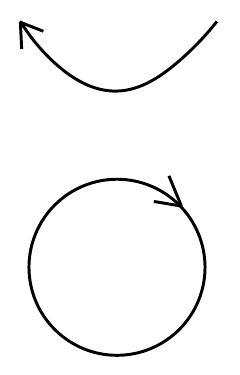}} \right] 
	-\left[ \raisebox{20 pt}{\includegraphics[angle = 180,scale=0.3]{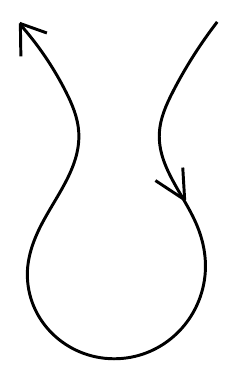}} \right]  \right)
	+ \left[ \raisebox{-15 pt}{\includegraphics[ scale=0.3]{kinkpos.pdf}} \right] 
	+ \left[ \raisebox{-15 pt}{\includegraphics[scale=0.3]{kinkpos1.pdf}} \right] \\
	&=& (q-q^{-1})\left ( q^{n-1}[n] \left [\raisebox{10pt}{\includegraphics[angle=180, scale=0.5]{line2.pdf}} \right] - \left [\raisebox{10pt}{\includegraphics[angle=180, scale=0.5]{line2.pdf}} \right]\right)\\
	 &+& q \left[ \raisebox{10pt}{\includegraphics[angle=180, scale=0.5]{line2.pdf}} \right] + q^{2n-1}\left[ \raisebox{10pt}{\includegraphics[angle=180, scale=0.5]{line1.pdf}} \right] \\
	 &=& q^{2n-1} \left[ \raisebox{10pt}{\includegraphics[angle=180, scale=0.5]{line2.pdf}} \right] +  q^{2n-1} \left[ \raisebox{10pt}{\includegraphics[angle=180, scale=0.5]{line1.pdf}} \right] = q^{2n-1} \left\llbracket  \raisebox{-3pt}{\includegraphics[width=.45in]{arc2}} \,\,\right\rrbracket .
	 \end{eqnarray*}
Similarly, one can show that 
$\left\llbracket \raisebox{-15 pt}{\includegraphics[scale=0.3]{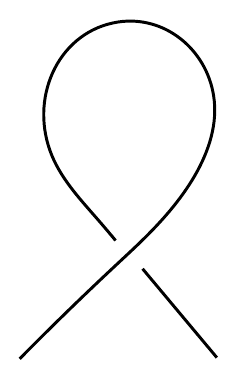}} \right\rrbracket 
= q^{1-2n} \left\llbracket \raisebox{-5pt}{\includegraphics[width=.45in]{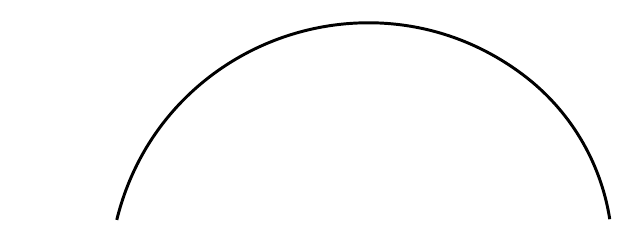}}\,\, \right\rrbracket$. 

In remains to show that $\left\llbracket \, \cdot \, \right\rrbracket$ is a regular isotopy invariant for unoriented links.
\begin{eqnarray*}
\left\llbracket \raisebox{17 pt}{\includegraphics[,angle = 180,scale=.25]{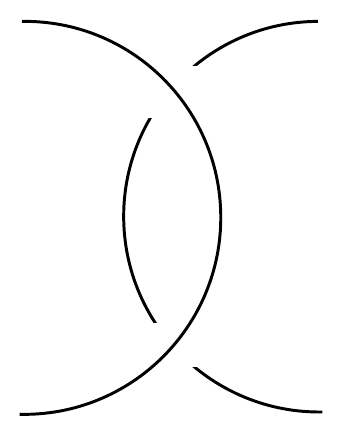}} \right\rrbracket
&=& (q-q^{-1}) \left( \left[ \raisebox{-15 pt}{\includegraphics[scale=.3]{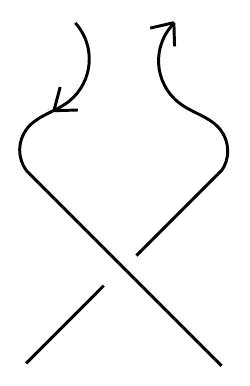}} \right] -
	\left[ \raisebox{-15 pt}{\includegraphics[scale=.3]{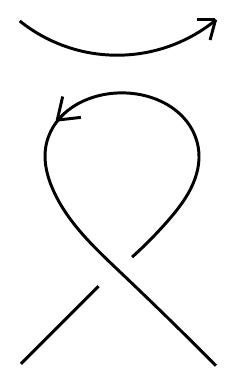}} \right] + 
	\left[ \raisebox{-15 pt}{\includegraphics[scale=.3]{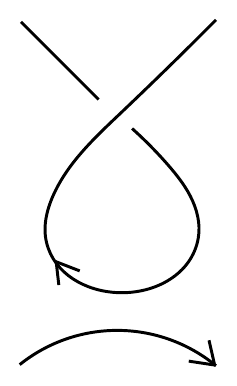}} \right] -
	\left[ \raisebox{-15 pt}{\includegraphics[scale=.3]{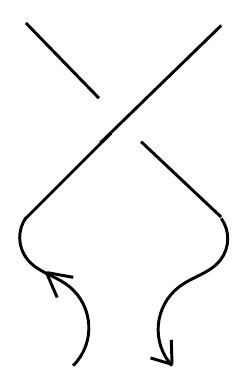}} \right] \right) \\
&&	+\left[ \raisebox{-15 pt}{\includegraphics[scale=.3]{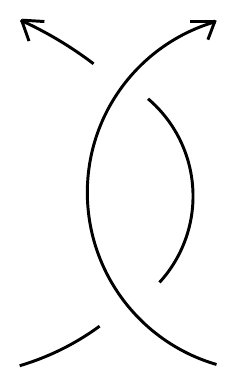}} \right] 
	+ \left[ \raisebox{-15 pt}{\includegraphics[scale=.3]{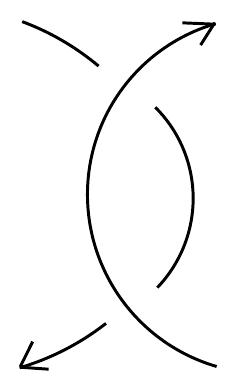}} \right] 
	+\left[ \raisebox{-15 pt}{\includegraphics[scale=.3]{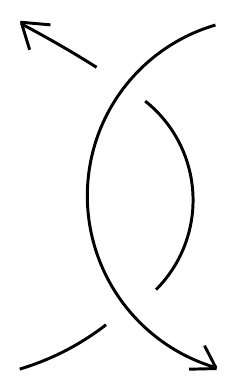}} \right] 
	+\left[ \raisebox{-15 pt}{\includegraphics[scale=.3]{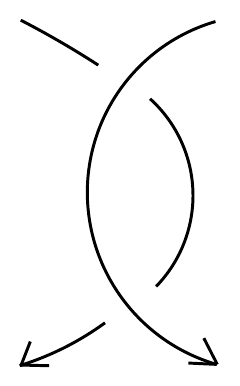}} \right] \\
& = & (q-q^{-1}) \left( \left[ \raisebox{15 pt}{\includegraphics[angle =270, scale=0.37]{crossingpos.pdf}} 			\right] 
	-q\left[ \raisebox{-10 pt}{\includegraphics[scale=0.37]{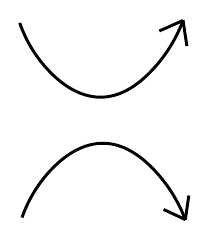}} \right] 
	+q^{-1} \left[ \raisebox{-10 pt}{\includegraphics[scale=0.37]{split2or1.pdf}} \right] 
	- \left[ \raisebox{15 pt}{\includegraphics[angle =270, scale=0.37]{crossingneg.pdf}} \right] 		\right) \\
 &&	+\left[ \raisebox{-10 pt}{\includegraphics[angle=90, scale=.4]{split2or1}} \right] 
	+ \left[ \raisebox{-10 pt}{\includegraphics[angle=90, scale=.4]{split2or}} \right] 
	+\left[ \raisebox{-10 pt}{\includegraphics[scale=.4]{split1or}} \right] 
	+\left[ \raisebox{14 pt}{\includegraphics[angle=270, scale=.4]{split2or1}} \right]  \\
& = & (q-q^{-1}) \left( \left[ \raisebox{15 pt}{\includegraphics[angle =270, scale=0.37]{crossingpos.pdf}} 			\right] 
	- \left[ \raisebox{15 pt}{\includegraphics[angle =270, scale=0.37]{crossingneg.pdf}} \right]
	-q\left[ \raisebox{-10 pt}{\includegraphics[scale=0.37]{split2or1.pdf}} \right] 
	+q^{-1} \left[ \raisebox{-10 pt}{\includegraphics[scale=0.37]{split2or1.pdf}} \right] \right) 
+ \left\llbracket \raisebox{-9 pt}{\includegraphics[scale=.37]{split1}} \right\rrbracket\\
& = & \left\llbracket \raisebox{-9 pt}{\includegraphics[scale=.37]{split1}} \right\rrbracket, \, \text{by the Conway identity for $[ \, \cdot \,]$}.
\end{eqnarray*}

The invariance of $\left\llbracket \, \cdot \, \right\rrbracket$ under the Reidemeister III move is verified in a similar fashion, and we leave the details to the reader.
\end{proof}

\section{The $SO(2n)$ Kauffman polynomial via planar $4$-valent graphs}\label{sec:main}

We seek to construct a state summation model for the $SO(2n)$ Kauffman polynomial, that works in much the same way as the MOY model works for the $sl(n)$ polynomial. Moreover, we want to derive such a state model by implementing the MOY construction into Jaeger's theorem. Therefore, the states corresponding to an unoriented link diagram $L$ will be unoriented 4-valent graphs obtained by resolving a crossing of $L$ in one of the following ways:
\[  \raisebox{-8pt}{\includegraphics[height=.3in]{split2}}\,,\quad \raisebox{-8pt}{\includegraphics[height=.3in]{split1}}\,,\quad \raisebox{-8pt}{\includegraphics[height=.3in]{vertnc}}\]
and we want to find some $A,B,C \in \mathbb{Z}[q,q^{-1}]$, such that

\begin{eqnarray}\label{eq:skeinrel}
 \left \llbracket  \,\raisebox{-8pt}{\includegraphics[height=.3in]{crossing} }\right \rrbracket = 
	 A \left \llbracket \, \raisebox{-8pt}{\includegraphics[height=.3in]{split2} }\right \rrbracket +
	 B \left \llbracket \,\raisebox{-8pt}{\includegraphics[height=.3in]{split1} }\right \rrbracket +
	 C \left \llbracket \, \raisebox{-8pt}{\includegraphics[height=.3in]{vertnc} }\right \rrbracket  .  
\end{eqnarray}
The state model that we wish to construct requires a consistent method to evaluate closed, unoriented $4$-valent graphs (the states associated with $L$). 

To this end, we note that by implementing the MOY state summation into Jaeger's model requires the bracket evaluation $[\Gamma]$, where $\Gamma$ is an oriented $4$-valent planar graph whose vertices are crossing-type oriented. We define
\begin{eqnarray}\label{eq:bracket}
 [\Gamma] := (q^{1-n})^{\textrm{rot} (\Gamma)}R(\Gamma), 
 \end{eqnarray}
where $\textrm{rot} (\Gamma)$, the \textit{rotation number} of such a graph $\Gamma$, is the sum of the rotation numbers of the disjoint oriented circles obtained by splicing each vertex of $\Gamma$ according to the orientation of its edges:
 \[\raisebox{-8pt}{\includegraphics[height=0.3in]{vertex}} \longrightarrow \raisebox{-8pt}{\includegraphics[height=0.3in]{orienres}}. \]  

We will regard the equation ~\eqref{eq:bracket} as a skein relation, as explained below:
\begin{eqnarray}\label{eq:skein-bracket}
 && \left[\, \raisebox{-8pt}{\includegraphics[height=0.3in]{vertex}}\, \right] = (q^{1-n})^{\textrm{rot} \left (\, \raisebox{-8pt}{\includegraphics[height=0.25in]{vertex}} \,\right )}R \left(\, \raisebox{-8pt}{\includegraphics[height=0.3in]{vertex}}\, \right) = (q^{1-n})^{\textrm{rot} \left (\, \raisebox{-8pt}{\includegraphics[height=0.25in]{orienres}} \,\right )}R \left(\, \raisebox{-8pt}{\includegraphics[height=0.3in]{vertex}}\, \right). 
  \end{eqnarray}

Jaeger's theorem implies that
$ \left \llbracket \,\raisebox{-8pt}{\includegraphics[height=.3in]{split2} }\right \rrbracket = \left [ \raisebox{-8pt}{\includegraphics[height=.3in]{splita}} \right ] 
		+ \left [ \raisebox{-8pt}{\includegraphics[height=.3in]{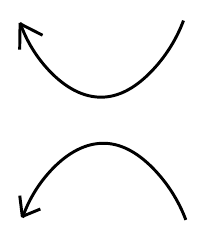}} \right ]
		+ \left [ \raisebox{-8pt}{\includegraphics[height=.3in]{split2or}} \right ]
		+ \left [ \raisebox{-8pt}{\includegraphics[height=.3in]{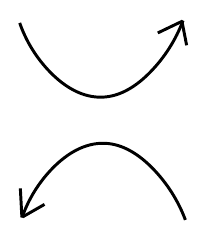}} \right ],  $
and to have a consistent construction, the evaluation $\left \llbracket \,\raisebox{-8pt}{\includegraphics[height=.3in]{vertnc} }\right \rrbracket$ will contain the bracket evaluations $\left[\, \raisebox{-8pt}{\includegraphics[height=0.3in]{vertex}}\, \right]$ for all such orientations of the vertex.

To determine what the coefficients $A,B,$ and $C$ must be, we compute $\left \llbracket \,\raisebox{-8pt}{\includegraphics[height=.3in]{crossing} }\right \rrbracket$ via Jaeger's model, and throughout the process, we evaluate the resulting oriented link diagrams using the MOY construction for the $sl(n)$ polynomial, $R$.
\begin{eqnarray*}
	\left \llbracket \,\raisebox{-8pt}{\includegraphics[height=.3in]{crossing} }\right \rrbracket 
	&=& (q-q^{-1}) \left ( \left [\raisebox{-8pt}{\includegraphics[height=.3in]{split2or}} \right ] 
		- \left [\raisebox{-8pt}{\includegraphics[height=.3in]{split1or}} \right ] \right ) 
		+ \left [\raisebox{-8pt}{\includegraphics[height=.3in]{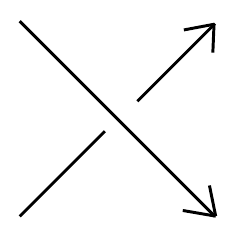}} \right ]
		+ \left [\raisebox{-8pt}{\includegraphics[height=.3in]{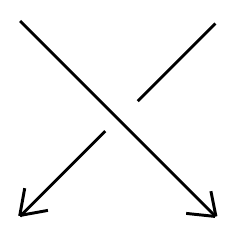}} \right ]
		+ \left [\raisebox{-8pt}{\includegraphics[height=.3in]{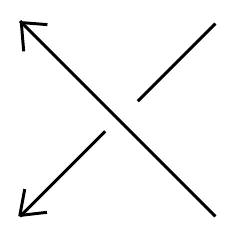}} \right ]
		+ \left [\raisebox{-8pt}{\includegraphics[height=.3in]{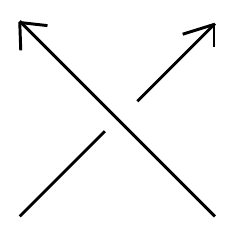}} \right ]\\
	&=& (q-q^{-1}) \left ( \left [\raisebox{-8pt}{\includegraphics[height=.3in]{split2or}} \right ] 
		- \left [\raisebox{-8pt}{\includegraphics[height=.3in]{split1or}} \right ] \right )\\
		&&+ (q^{1-n})^{\textrm{rot} \left (\raisebox{-5pt}{\includegraphics[height=.2in]{crossinga}}\right )}
			R \left (\raisebox{-8pt}{\includegraphics[height=.3in]{crossinga}}\right )	
		+ (q^{1-n})^{\textrm{rot} \left (\raisebox{-5pt}{\includegraphics[height=.2in]{crossingb}}\right )}
			R \left (\raisebox{-8pt}{\includegraphics[height=.3in]{crossingb}}\right )\\
 		&&+ (q^{1-n})^{\textrm{rot} \left (\raisebox{-5pt}{\includegraphics[height=.2in]{crossingc}}\right )}
			R \left (\raisebox{-8pt}{\includegraphics[height=.3in]{crossingc}}\right )	
		+ (q^{1-n})^{\textrm{rot} \left (\raisebox{-5pt}{\includegraphics[height=.2in]{crossingd}}\right )}
			R \left (\raisebox{-8pt}{\includegraphics[height=.3in]{crossingd}}\right ).	 
\end{eqnarray*} 
Employing the skein relations in Figure~\ref{fig:decomposition}, we have
\begin{eqnarray*}
\left \llbracket \,\raisebox{-8pt}{\includegraphics[height=.3in]{crossing} }\right \rrbracket 
	&=& (q-q^{-1}) \left ( \left [\raisebox{-8pt}{\includegraphics[height=.3in]{split2or}} \right ] 
		- \left [\raisebox{-8pt}{\includegraphics[height=.3in]{split1or}} \right ] \right ) \\
		&&+ (q^{1-n})^{\textrm{rot} \left (\raisebox{-5pt}{\includegraphics[height=.2in]{splita}}\right )}
			\left (qR\left (\raisebox{-8pt}{\includegraphics[height=.3in]{splita}}\right ) 
			- R\left (\raisebox{-8pt}{\includegraphics[height=.3in]{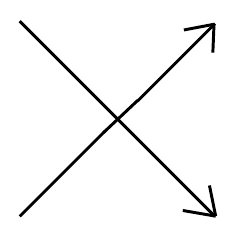}} \right ) \right )	\\
		&&+ (q^{1-n})^{\textrm{rot} \left (\raisebox{-5pt}{\includegraphics[height=.2in]{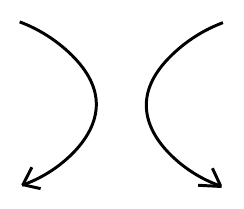}}\right )}
			\left (q^{-1} R\left (\raisebox{-8pt}{\includegraphics[height=.3in]{splitb}}\right ) 
			- R\left (\raisebox{-8pt}{\includegraphics[height=.3in]{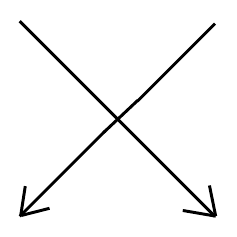}} \right ) \right ) \\
		&&+ (q^{1-n})^{\textrm{rot} \left (\raisebox{-5pt}{\includegraphics[height=.2in]{splitc}}\right )}
			\left (qR\left (\raisebox{-8pt}{\includegraphics[height=.3in]{splitc}}\right ) 
			- R\left (\raisebox{-8pt}{\includegraphics[height=.3in]{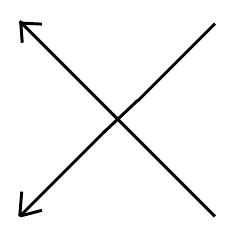}} \right ) \right ) \\
		&&+ (q^{1-n})^{\textrm{rot} \left (\raisebox{-5pt}{\includegraphics[height=.2in]{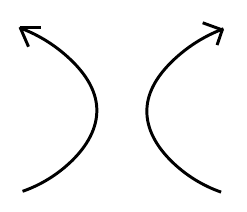}}\right )}
			\left (q^{-1} R\left (\raisebox{-8pt}{\includegraphics[height=.3in]{splitd}}\right ) 
			- R\left (\raisebox{-8pt}{\includegraphics[height=.3in]{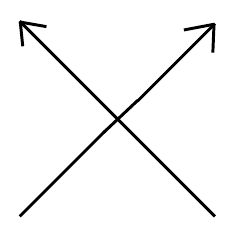}} \right ) \right ). 
\end{eqnarray*}			
Making use of the skein relation~\eqref{eq:skein-bracket}, we obtain
\begin{eqnarray*}
\left \llbracket \,\raisebox{-8pt}{\includegraphics[height=.3in]{crossing} }\right \rrbracket 
	&=& (q-q^{-1})\left ( \left [\raisebox{-8pt}{\includegraphics[height=.3in]{split2or}} \right ] 
		- \left [\raisebox{-8pt}{\includegraphics[height=.3in]{split1or}} \right ] \right )\\
		&&+ q\left [\raisebox{-8pt}{\includegraphics[height=.3in]{splita}} \right ] 
		- \left [\raisebox{-8pt}{\includegraphics[height=.3in]{verta}} \right ] 
		+ q^{-1} \left [\raisebox{-8pt}{\includegraphics[height=.3in]{splitb}} \right ] 
		- \left [\raisebox{-8pt}{\includegraphics[height=.3in]{vertb}} \right ] \\
		&&+ q \left [\raisebox{-8pt}{\includegraphics[height=.3in]{splitc}} \right ] 
		- \left [\raisebox{-8pt}{\includegraphics[height=.3in]{vertc}} \right ]
		+ q^{-1}\left [\raisebox{-8pt}{\includegraphics[height=.3in]{splitd}} \right ] 
		- \left [\raisebox{-8pt}{\includegraphics[height=.3in]{vertd}} \right ] \\
	&=& q \left ( \left [ \raisebox{-8pt}{\includegraphics[height=.3in]{splita}} \right ] 
		+ \left [ \raisebox{-8pt}{\includegraphics[height=.3in]{splitc}} \right ]
		+ \left [ \raisebox{-8pt}{\includegraphics[height=.3in]{split2or}} \right ]
		+ \left [ \raisebox{-8pt}{\includegraphics[height=.3in]{split2ro}} \right ] \right )\\
		&&+  q^{-1} \left ( \left [ \raisebox{-8pt}{\includegraphics[height=.3in]{splitb}} \right ] 
		+ \left [ \raisebox{-8pt}{\includegraphics[height=.3in]{splitd}} \right ]
		+ \left [ \raisebox{-8pt}{\includegraphics[height=.3in]{split1or}} \right ]
		+ \left [ \raisebox{-8pt}{\includegraphics[height=.3in]{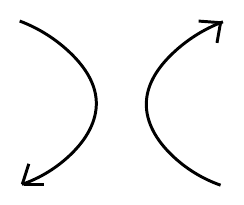}} \right ] \right )\\
		&& - q\left [\raisebox{-8pt}{\includegraphics[height=.3in]{split2ro}}\right ] 
		- q\left[\raisebox{-8pt}{\includegraphics[height=.3in]{split1or}}\right ] - q^{-1}\left [\raisebox{-8pt}{\includegraphics[height=.3in]{split1ro}}\right ] 
		- q^{-1}\left[\raisebox{-8pt}{\includegraphics[height=.3in]{split2or}}\right ] \\
		&& -\left ( \left [\raisebox{-8pt}{\includegraphics[height=.3in]{verta}}\right ]
		+ \left [\raisebox{-8pt}{\includegraphics[height=.3in]{vertb}}\right ]
		+ \left [\raisebox{-8pt}{\includegraphics[height=.3in]{vertc}}\right ]
		+ \left [\raisebox{-8pt}{\includegraphics[height=.3in]{vertd}}\right ] \right ).
		\end{eqnarray*}
Therefore, we have	
\begin{eqnarray*}
	\left \llbracket\, \raisebox{-8pt}{\includegraphics[height=.3in]{crossing} }\right \rrbracket 
	&=& q \left \llbracket \,\raisebox{-8pt}{\includegraphics[height=.3in]{split2}}\, \right \rrbracket
		+ q^{-1} \left \llbracket \,\raisebox{-8pt}{\includegraphics[height=.3in]{split1}}\, \right \rrbracket \\
		&& - q\left [\raisebox{-8pt}{\includegraphics[height=.3in]{split2ro}}\right ] 
		- q\left[\raisebox{-8pt}{\includegraphics[height=.3in]{split1or}}\right ] - q^{-1}\left [\raisebox{-8pt}{\includegraphics[height=.3in]{split1ro}}\right ] 
		- q^{-1}\left[\raisebox{-8pt}{\includegraphics[height=.3in]{split2or}}\right ]\\
		&&-\left ( \left [\raisebox{-8pt}{\includegraphics[height=.3in]{verta}}\right ]
		+ \left [\raisebox{-8pt}{\includegraphics[height=.3in]{vertb}}\right ]
		+ \left [\raisebox{-8pt}{\includegraphics[height=.3in]{vertc}}\right ]
		+ \left [\raisebox{-8pt}{\includegraphics[height=.3in]{vertd}}\right ] \right ).
		\end{eqnarray*}

Comparing the last equality with equation~\eqref{eq:skeinrel}, we see that in order to work with a certain evaluation $\left \llbracket\, \raisebox{-6pt}{\includegraphics[height=.25in]{vertnc} }\right \rrbracket $ for an unoriented vertex, we must also take in consideration \textit{alternating orientations} for edges meeting at a vertex, and define the \textit{bracket} of an \textit{alternating oriented vertex} as follows:
\begin{eqnarray}
\left [\raisebox{-8pt}{\includegraphics[height=.3in]{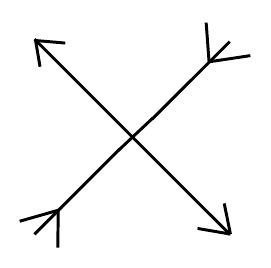}} \right ] := q\left [\raisebox{-8pt}{\includegraphics[height=.3in]{split1or}} \right ] + q^{-1} \left [\raisebox{-8pt}{\includegraphics[height=.3in]{split2or}} \right ].
\end{eqnarray}

The above computations also imply the need of the following definition:
\[
 \left \llbracket \,\raisebox{-8pt}{\includegraphics[height=.3in]{vertnc} }\right \rrbracket := \left [\raisebox{-8pt}{\includegraphics[height=.3in]{verta}} \right ] + \left [\raisebox{-8pt}{\includegraphics[height=.3in]{vertb}} \right ] + \left [\raisebox{-8pt}{\includegraphics[height=.3in]{vertc}} \right ] + \left [\raisebox{-8pt}{\includegraphics[height=.3in]{vertd}} \right ] + \left [\raisebox{-8pt}{\includegraphics[height=.3in]{hyp1}} \right ] + \left [\raisebox{-8pt}{\includegraphics[height=.3in]{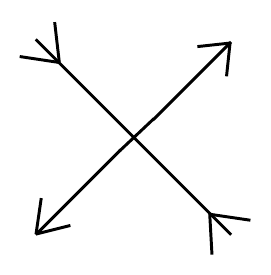}} \right ].
\]
Implementing the above definitions into our previous computations, we obtain

\begin{eqnarray} \label{eqn:skein-def}
 \left \llbracket \,\raisebox{-8pt}{\includegraphics[height=.3in]{crossing} }\right \rrbracket  = q \left \llbracket\, \raisebox{-8pt}{\includegraphics[height=.3in]{split2} }\right \rrbracket
		+ q^{-1} \left \llbracket \,\raisebox{-8pt}{\includegraphics[height=.3in]{split1} }\right \rrbracket 
	 	-  \left \llbracket \,\raisebox{-8pt}{\includegraphics[height=.3in]{vertnc} }\right \rrbracket .
\end{eqnarray} 
Therefore, $A=q, B=q^{-1}, \, \text{and}\, \, C=-1$.

We have seen that the implementation of the MOY state model into Jaeger's state summation requires \textit{balanced} oriented $4$-valent graphs (in the sense that the total degree of a vertex is zero), with vertices being either crossing-type oriented or alternating oriented.


\begin{proposition}\label{R0}
The following identity holds:
\[\left \llbracket \raisebox{-8pt}{\includegraphics[height=.3in]{unknot1}} \right \rrbracket = [2n-1]+1.\]
\end{proposition}

\begin{proof}This identity holds by Jaeger's theorem.
 \end{proof}


\begin{proposition}\label{R1}
 The following graph skein relation holds:
\[
\left \llbracket \raisebox{-10pt}{\includegraphics[height=.35in]{kinked}} \,\right \rrbracket = ([2n-2]+[2])\left \llbracket \raisebox{-3pt}{\includegraphics[width=.4in]{arc2}}\, \right \rrbracket . \] 
\end{proposition}
\begin{proof}

\[
\left \llbracket \raisebox{-10pt}{\includegraphics[height=.35in]{kinked}}\, \right \rrbracket = \left [ \raisebox{-10pt}{\includegraphics[height=.35in]{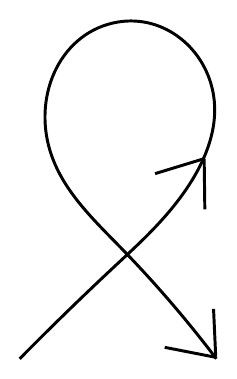}}\, \right ] + \left [ \raisebox{-10pt}{\includegraphics[height=.35in]{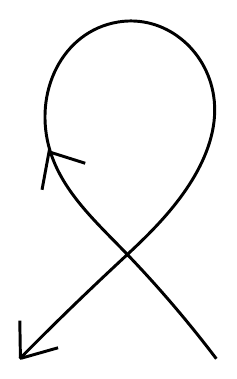}}\, \right ] + \left [ \raisebox{-10pt}{\includegraphics[height=.35in]{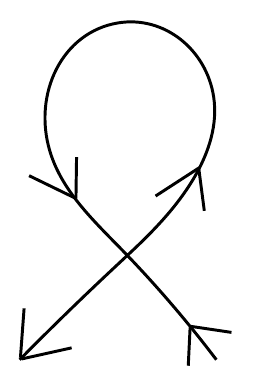}}\, \right ] + \left [ \raisebox{-10pt}{\includegraphics[height=.35in]{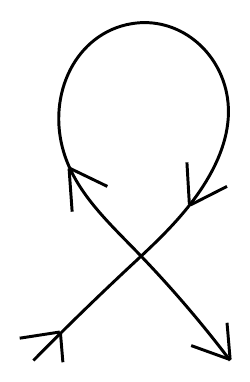}}\, \right ].
\]
Now, for the first oriented diagram, we have
\begin{align*}
\left [ \raisebox{-10pt}{\includegraphics[height=.35in]{kinkeda}}\, \right ] &= q^{(1-n)\textrm{rot}\left (\raisebox{-6pt}{\includegraphics[height=.25in]{kinkeda}}\right )} R\left (\raisebox{-10pt}{\includegraphics[height=.35in]{kinkeda}}\right ) 
= q^{(1-n)\textrm{rot}\left (\raisebox{-6pt}{\includegraphics[height=.25in]{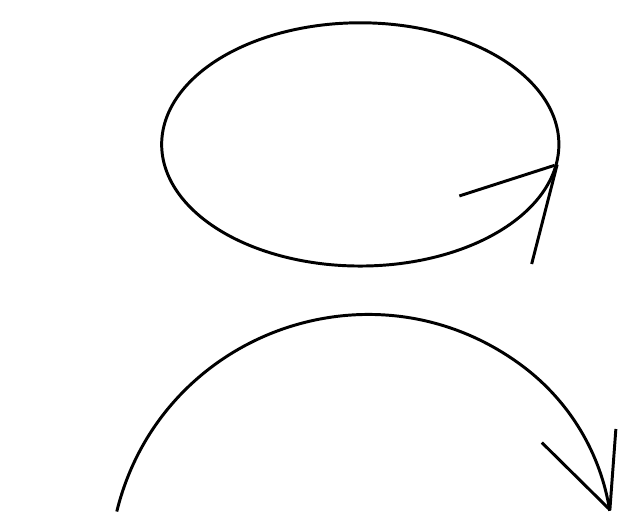}}\right )} [n-1] R\left (\raisebox{-3pt}{\includegraphics[width=.4in]{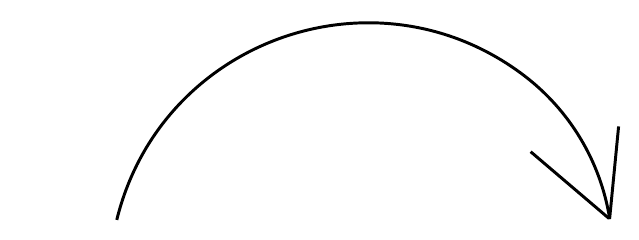}}\,\right )\\
&= q^{1-n}[n-1]q^{(1-n)\textrm{rot} \left (\raisebox{-1pt}{\includegraphics[width=.2in]{arc2a}}\, \right )} R \left (\raisebox{-3pt}{\includegraphics[width=.4in]{arc2a}}\, \right ) = q^{1-n}[n-1]\left [ \raisebox{-3pt}{\includegraphics[width=.4in]{arc2a}}\, \right ],
\end{align*}
and for the third oriented diagram, we have
\begin{align*}
\left [ \raisebox{-8pt}{\includegraphics[height=.35in]{kinkedc}} \right ]
	&= q \left [ \raisebox{-8pt}{\includegraphics[height=.3in]{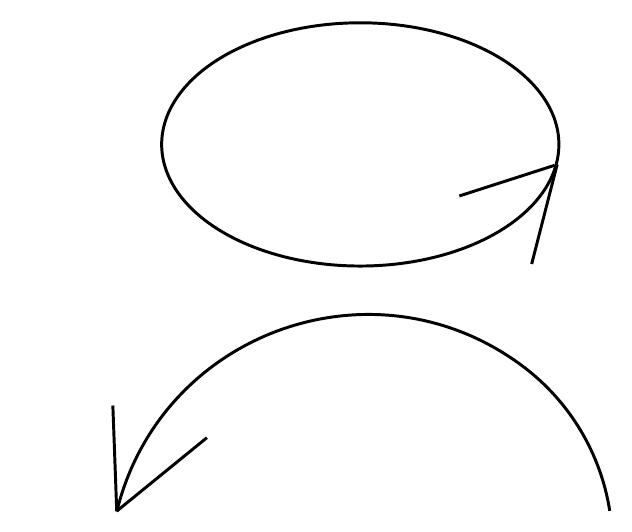}}\, \right ]
		+ q^{-1}\left [ \raisebox{-3pt}{\includegraphics[width=.4in]{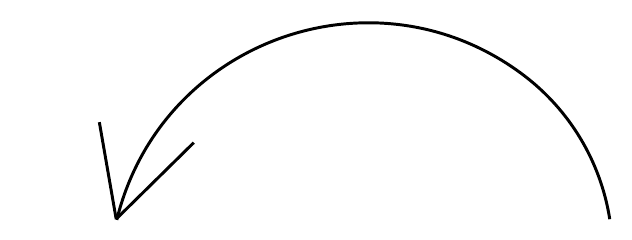}}\, \right ] \\
	&= q\cdot q^{(1-n)\textrm{rot}\left (\raisebox{-6pt}{\includegraphics[height=.25in]{arcirb}}\right )} R\left (\raisebox{-8pt}
			{\includegraphics[height=.3in]{arcirb}}\,\right )
		+ q^{-1}\left [ \raisebox{-3pt}{\includegraphics[width=.4in]{arc2b}}\, \right ] \\
	&= q\cdot q^{1-n} \cdot q^{(1-n)\textrm{rot}\left (\raisebox{-1pt}{\includegraphics[width=.2in]{arc2b}}\right )} [n] 
			R\left ( \raisebox{-3pt}{\includegraphics[width=.4in]{arc2b}} \, \right )
		+ q^{-1}\left [ \raisebox{-3pt}{\includegraphics[width=.4in]{arc2b}} \, \right ] \\
	&= q^{2-n} [n] \left [ \raisebox{-3pt}{\includegraphics[width=.4in]{arc2b}} \,\right ] 
		+ q^{-1} \left [ \raisebox{-3pt}{\includegraphics[width=.4in]{arc2b}}\, \right ]
		= (q^{2-n}[n]+q^{-1})\left [ \raisebox{-3pt}{\includegraphics[width=.4in]{arc2b}}\, \right ].
\end{align*}
Similarly, we obtain 
\[ \left [ \raisebox{-8pt}{\includegraphics[height=.35in]{kinkedb}}\, \right ] = q^{n-1}[n-1]\left [\raisebox{-3pt}{\includegraphics[width=.4in]{arc2b}}\, \right ] \,\, \text{and} \,\,
 \left [ \raisebox{-8pt}{\includegraphics[height=.35in]{kinkedd}} \right ] = (q^{n-2}[n]+q)\left [\raisebox{-3pt}{\includegraphics[width=.4in]{arc2a}} \,\right ] . \]  
Using these evaluations for each of the oriented states, we arrive at 
\[ \left \llbracket \raisebox{-8pt}{\includegraphics[height=.3in]{kinked}}\, \right \rrbracket = ([2n-2]+[2]) \left( \left [\raisebox{-3pt}{\includegraphics[width=.4in]{arc2b}}\, \right ] + \left [\raisebox{-3pt}{\includegraphics[width=.4in]{arc2a}}\, \right ] \right) = ([2n-2]+[2])\left \llbracket \raisebox{-2pt}{\includegraphics[width=.3in]{arc2}}\, \right \rrbracket . \]
\end{proof}


\begin{proposition} \label{R2} The following skein relation holds:
\[\left \llbracket\, \raisebox{-8pt}{\includegraphics[height=.3in]{R2}} \,\right \rrbracket = ([2n-3]+1)\left \llbracket\, \raisebox{-8pt}{\includegraphics[height=.3in]{split2}}\, \right \rrbracket + [2]\left \llbracket\, \raisebox{-8pt}{\includegraphics[height=.3in]{vertnc}}\, \right \rrbracket .\]
\end{proposition}
\begin{proof} We know that
\begin{align*}
\left \llbracket \raisebox{-8pt}{\includegraphics[height=.3in]{R2}} \right \rrbracket 
	&=  \left [\raisebox{-8pt}{\includegraphics[height=.3in]{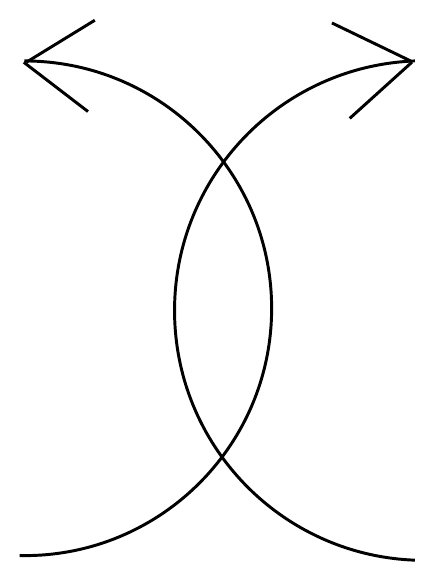}} \right ]
		+\left [\raisebox{-8pt}{\includegraphics[height=.3in]{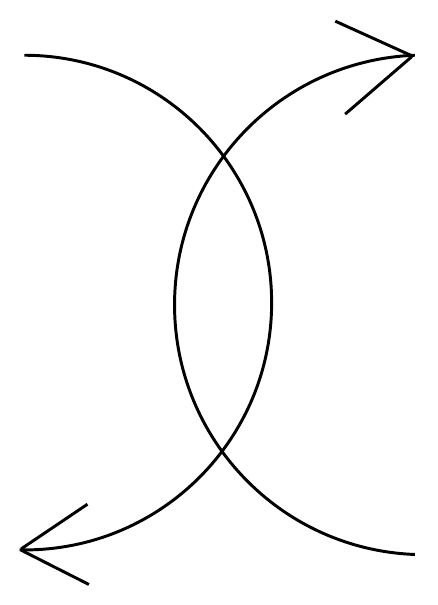}} \right ]
		+\left [\raisebox{-8pt}{\includegraphics[height=.3in]{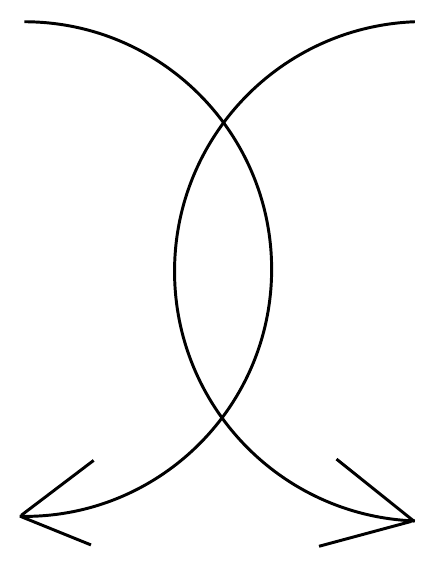}} \right ]
		+\left [\raisebox{-8pt}{\includegraphics[height=.3in]{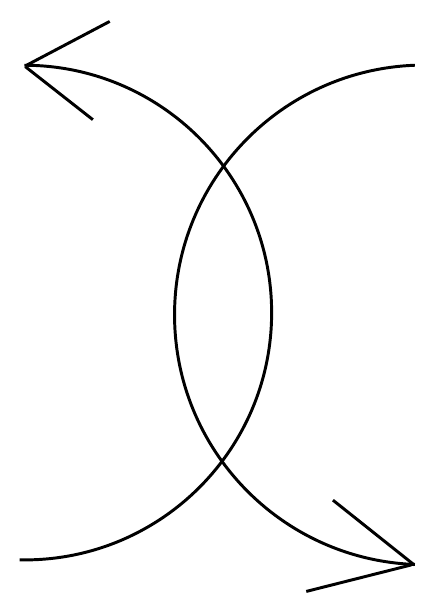}} \right ]
		+\left [\raisebox{-8pt}{\includegraphics[height=.3in]{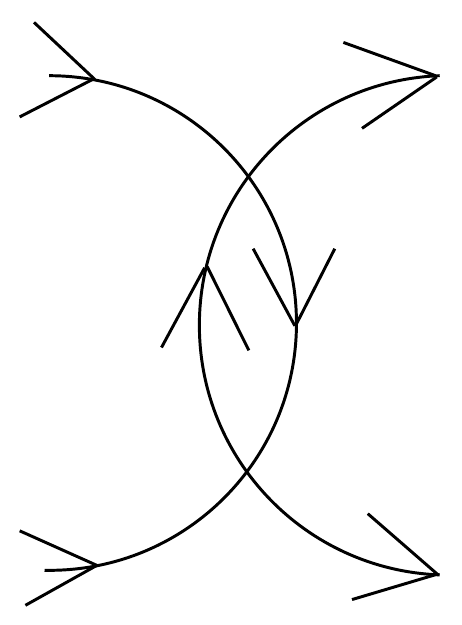}} \right ]
		+\left [\raisebox{-8pt}{\includegraphics[height=.3in]{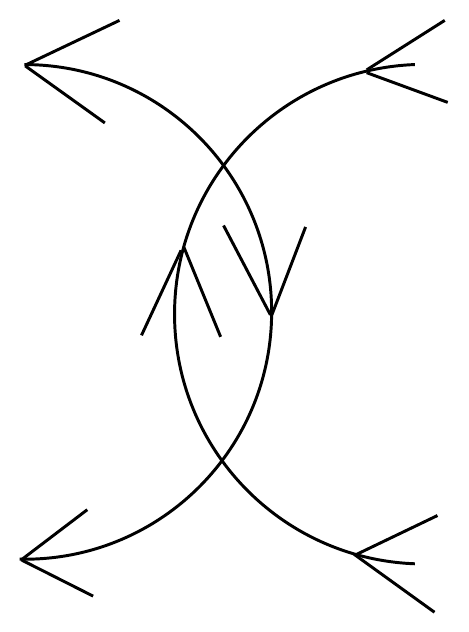}} \right ]
		+\left [\raisebox{-8pt}{\includegraphics[height=.3in]{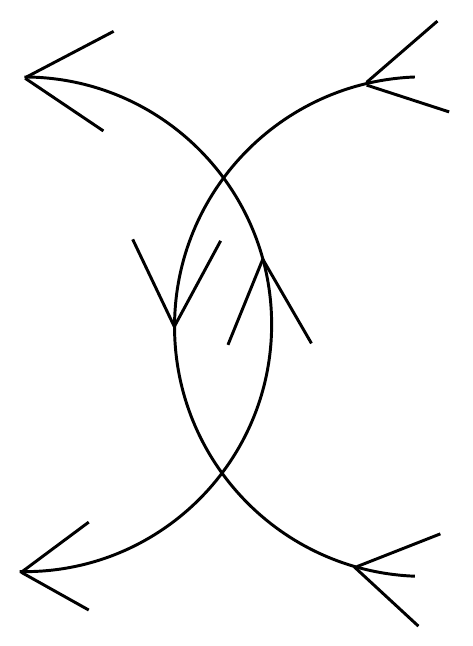}} \right ]
		+\left [\raisebox{-8pt}{\includegraphics[height=.3in]{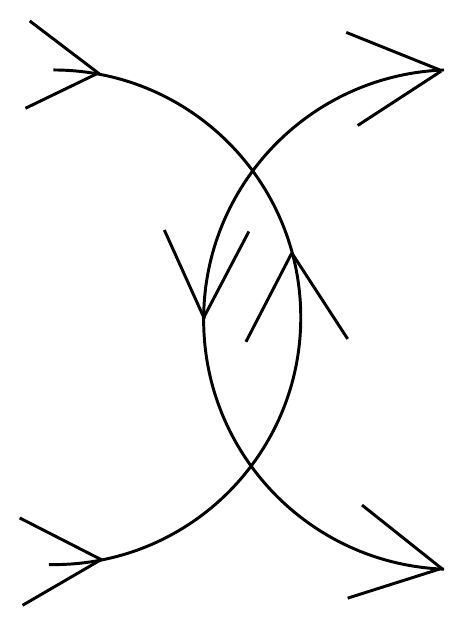}} \right ] \\
		&+\left [\raisebox{-8pt}{\includegraphics[height=.3in]{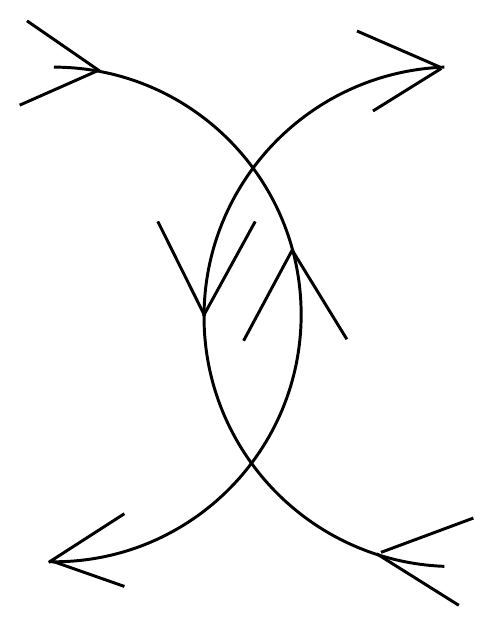}} \right ]  +\left [\raisebox{-8pt}{\includegraphics[height=.3in]{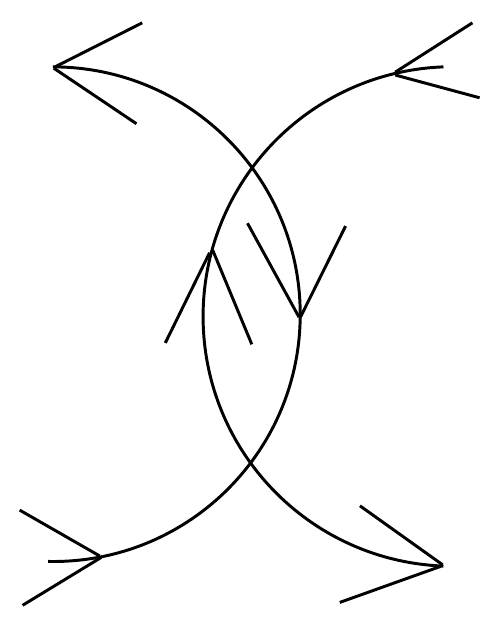}} \right ].
\end{align*}
Now, $\left [\raisebox{-8pt}{\includegraphics[height=.3in]{R2a}} \right ]
	=  (q^{1-n})^{\textrm{rot} \left (\raisebox{-5pt}{\includegraphics[height=.2in]{R2a}}\right )}
			R \left (\raisebox{-8pt}{\includegraphics[height=.3in]{R2a}}\right )
	=  (q^{1-n})^{\textrm{rot} \left (\raisebox{-5pt}{\includegraphics[height=.2in]{splitd}}\right )}
			[2] R \left (\raisebox{-8pt}{\includegraphics[height=.3in]{vertd}}\right )
	= [2]\left [\raisebox{-8pt}{\includegraphics[height=.3in]{vertd}} \right ] $, \\
and 	
\begin{align*}
\left [\raisebox{-8pt}{\includegraphics[height=.3in]{R2b}} \right ] &= (q^{1-n})^{\textrm{rot} \left (\raisebox{-5pt}{\includegraphics[height=.2in]{R2b}}\right )}
			R \left (\raisebox{-8pt}{\includegraphics[height=.3in]{R2b}}\right )\\
			& = (q^{1-n})^{\textrm{rot} \left (\raisebox{-5pt}{\includegraphics[height=.2in]{split1ro}}\right )} \left( R \left (\raisebox{-8pt}{\includegraphics[height=.3in]{split1ro}} \right) + [n-2] R \left (\raisebox{-8pt}{\includegraphics[height=.3in]{split2ro}} \right)
			 \right )\\
			 & =  \left [\raisebox{-8pt}{\includegraphics[height=.3in]{split1ro}} \right ] 
		+q^{n-1}[n-2]\left [\raisebox{-8pt}{\includegraphics[height=.3in]{split2ro}} \right ], 
\end{align*}
where we used the fact that $\textrm{rot} \left (\raisebox{-5pt}{\includegraphics[height=.2in]{split1ro}}\right ) = \textrm{rot} \left (\raisebox{-5pt}{\includegraphics[height=.2in]{split2ro}}\right )-1$.
We also have that
\begin{align*}
\left [\raisebox{-8pt}{\includegraphics[height=.3in]{R2c}} \right ]
	&= q\left [\raisebox{-8pt}{\includegraphics[height=.3in]{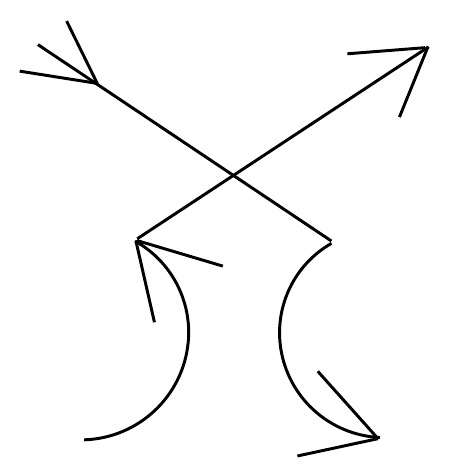}} \right ]
		+q^{-1} \left [\raisebox{-8pt}{\includegraphics[height=.3in]{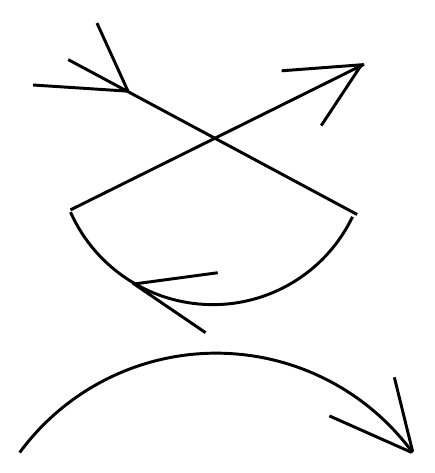}} \right ]
	= q\left [\raisebox{-8pt}{\includegraphics[height=.3in]{RtwoA}} \right ] +q^{-1}\dot (q^{1-n})^{\textrm{rot}  \left (\raisebox{-5pt}{\includegraphics[height=.2in]{splita}} - 1\right ) } R \left (\raisebox{-8pt}{\includegraphics[height=.3in]{RtwoB}} \right )\\
	&=  q\left [\raisebox{-8pt}{\includegraphics[height=.3in]{RtwoA}} \right ] + q^{-1}\dot q^{n-1} (q^{1-n})^{\textrm{rot}  \left (\raisebox{-5pt}{\includegraphics[height=.2in]{splita}}\right)} [n-1] R \left (\raisebox{-8pt}{\includegraphics[height=.3in]{splita}}  \right)
		\\
		& = q\left [\raisebox{-8pt}{\includegraphics[height=.3in]{verta}} \right ]
		+q^{n-2}[n-1]\left [\raisebox{-8pt}{\includegraphics[height=.3in]{splita}} \right ].
\end{align*}
Similarly, for the bigon with alternating oriented vertices, we have
\begin{align*}
\left [\raisebox{-8pt}{\includegraphics[height=.3in]{R2h}} \right ]
	&= q^{-1}\left [\raisebox{-8pt}{\includegraphics[height=.3in]{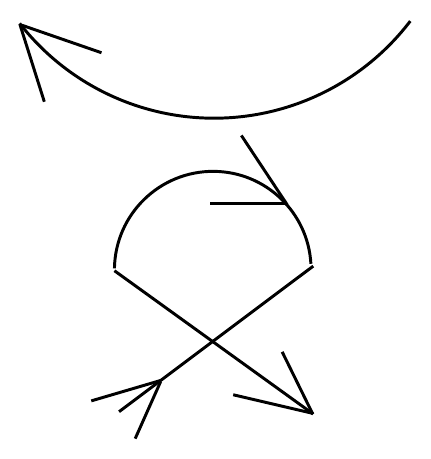}} \right ]
		+q\left [\raisebox{-8pt}{\includegraphics[height=.3in]{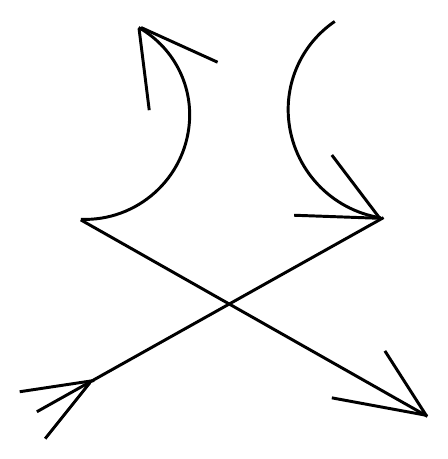}} \right ] \\
	&= q^{-1}\left ( q\left [\raisebox{-8pt}{\includegraphics[height=.3in]{split2or}} \right ] 
		+ q^{-1}\left [\raisebox{-8pt}{\includegraphics[height=.3in]{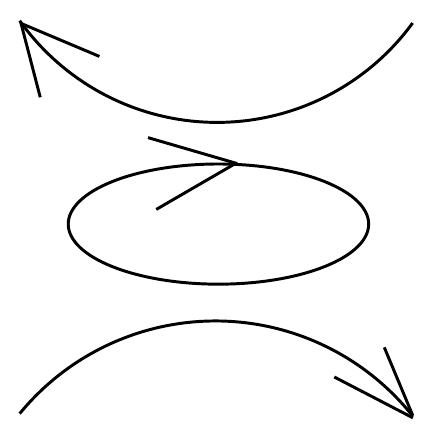}} \right ]\right )
		+q\left ( q\left [\raisebox{-8pt}{\includegraphics[height=.3in]{split1or}} \right ]
		+q^{-1}\left [\raisebox{-8pt}{\includegraphics[height=.3in]{split2or}} \right ] \right ) \\
	&=\left [\raisebox{-8pt}{\includegraphics[height=.3in]{split2or}} \right ]
		+q^{-2}\cdot q^{n-1}
			[n]\left [\raisebox{-8pt}{\includegraphics[height=.3in]{split2or}}\right ]
		+q^2\left [\raisebox{-8pt}{\includegraphics[height=.3in]{split1or}} \right ]
		+\left [\raisebox{-8pt}{\includegraphics[height=.3in]{split2or}} \right ] \\
		&=q^2\left [\raisebox{-8pt}{\includegraphics[height=.3in]{split1or}} \right ]
		+(q^{n-3}[n]+2)\left [\raisebox{-8pt}{\includegraphics[height=.3in]{split2or}} \right ].
\end{align*}
The remaining diagrams can be evaluated similarly. Thus, we have
\begin{eqnarray*}
\left \llbracket \raisebox{-8pt}{\includegraphics[height=.3in]{R2}} \right \rrbracket	
&=& \left [\raisebox{-8pt}{\includegraphics[height=.3in]{R2a}} \right ]
		+\left [\raisebox{-8pt}{\includegraphics[height=.3in]{R2b}} \right ]
		+\left [\raisebox{-8pt}{\includegraphics[height=.3in]{R2d}} \right ]
		+\left [\raisebox{-8pt}{\includegraphics[height=.3in]{R2e}} \right ]
		+\left [\raisebox{-8pt}{\includegraphics[height=.3in]{R2c}} \right ]
		+\left [\raisebox{-8pt}{\includegraphics[height=.3in]{R2g}} \right ]
		+\left [\raisebox{-8pt}{\includegraphics[height=.3in]{R2f}} \right ]
		+\left [\raisebox{-8pt}{\includegraphics[height=.3in]{R2i}} \right ] \\
		&&+\left [\raisebox{-8pt}{\includegraphics[height=.3in]{R2j}} \right ]  +\left [\raisebox{-8pt}{\includegraphics[height=.3in]{R2h}} \right ],
		 \end{eqnarray*}
and using the above computations yields
 \begin{eqnarray*}
 \left \llbracket \raisebox{-8pt}{\includegraphics[height=.3in]{R2}} \right \rrbracket
	&=& [2]\left [\raisebox{-8pt}{\includegraphics[height=.3in]{vertd}} \right ]
		+ \left ( \left [\raisebox{-8pt}{\includegraphics[height=.3in]{split1ro}} \right ] 
		+q^{n-1}[n-2]\left [\raisebox{-8pt}{\includegraphics[height=.3in]{split2ro}} \right ] \right )\\
		&& + [2]\left [\raisebox{-8pt}{\includegraphics[height=.3in]{vertb}} \right ] + \left ( \left [\raisebox{-8pt}{\includegraphics[height=.3in]{split1or}} \right ]
		+q^{1-n}[n-2]\left [\raisebox{-8pt}{\includegraphics[height=.3in]{split2or}} \right ] \right ) \\
		&&+\left ( q\left [\raisebox{-8pt}{\includegraphics[height=.3in]{verta}} \right ]
		+ q^{n-2}[n-1]\left [\raisebox{-8pt}{\includegraphics[height=.3in]{splita}} \right ] \right ) +\left ( q\left [\raisebox{-8pt}{\includegraphics[height=.3in]{vertc}} \right ]
		+q^{n-2}[n-1]\left [\raisebox{-8pt}{\includegraphics[height=.3in]{splitc}} \right ] \right )  \\
		&& + \left (q^{-1}\left [\raisebox{-8pt}{\includegraphics[height=.3in]{vertc}} \right ] + q^{2-n}[n-1]\left [\raisebox{-8pt}{\includegraphics[height=.3in]{splitc}} \right ] \right ) + \left (q^{-1}\left [\raisebox{-8pt}{\includegraphics[height=.3in]{verta}} \right ] + q^{2-n}[n-1]\left [\raisebox{-8pt}{\includegraphics[height=.3in]{splita}} \right ]\right ) \\
		&&+\left ( q^{-2}\left [\raisebox{-8pt}{\includegraphics[height=.3in]{split1ro}} \right ] + (q^{3-n}[n]+2)\left [\raisebox{-8pt}{\includegraphics[height=.3in]{split2ro}} \right ] \right ) + \left ( q^2\left [\raisebox{-8pt}{\includegraphics[height=.3in]{split1or}} \right ]
		+(q^{n-3}[n]+2)\left [\raisebox{-8pt}{\includegraphics[height=.3in]{split2or}} \right ] \right ).
		 \end{eqnarray*}
Combining like terms, we have
\begin{eqnarray*}
	\left \llbracket \raisebox{-8pt}{\includegraphics[height=.3in]{R2}} \right \rrbracket
	&=&(q+q^{-1})\left [\raisebox{-8pt}{\includegraphics[height=.3in]{hyp1}} \right ] 
		+(q+q^{-1})\left [\raisebox{-8pt}{\includegraphics[height=.3in]{hyp2}} \right ] 
		+[2]\left [\raisebox{-8pt}{\includegraphics[height=.3in]{vertd}} \right ] 
		+[2]\left [\raisebox{-8pt}{\includegraphics[height=.3in]{vertb}} \right ] 
		+(q_q^{-1})\left [\raisebox{-8pt}{\includegraphics[height=.3in]{verta}} \right ] \\
		&&+(q+q^{-1})\left [\raisebox{-8pt}{\includegraphics[height=.3in]{vertc}} \right ] 
		+([2n-3]+1)\left [\raisebox{-8pt}{\includegraphics[height=.3in]{split2ro}} \right ] 
		+([2n-3]+1)\left [\raisebox{-8pt}{\includegraphics[height=.3in]{splita}} \right ] \\
		&&+([2n-3]+1)\left [\raisebox{-8pt}{\includegraphics[height=.3in]{split2or}} \right ] 
		+([2n-3]+1)\left [\raisebox{-8pt}{\includegraphics[height=.3in]{splitc}} \right ] \\
	&=& [2]\left \llbracket \raisebox{-8pt}{\includegraphics[height=.3in]{vertnc} }\right \rrbracket + ([2n-3]+1)\left \llbracket \raisebox{-8pt}{\includegraphics[height=.3in]{split2} }\right \rrbracket ,
\end{eqnarray*}
which completes the proof.
\end{proof}

\begin{proposition} \label{R3}
The following graph skein relation holds:
\begin{eqnarray*}
&&\left \llbracket \raisebox{-8pt}{\includegraphics[height=.3in]{RIII1a}} \right \rrbracket
	+ \left \llbracket \raisebox{-8pt}{\includegraphics[height=.3in]{RIII1b}} \right \rrbracket
	- \left \llbracket \raisebox{-8pt}{\includegraphics[height=.3in]{RIII1c}} \right \rrbracket
	- \left \llbracket \raisebox{-8pt}{\includegraphics[height=.3in]{RIII1d}} \right \rrbracket
	- [2n-4] \left \llbracket \raisebox{-8pt}{\includegraphics[height=.3in]{RIII1e}} \right \rrbracket = \\
	&& \left \llbracket \raisebox{-8pt}{\includegraphics[height=.3in]{RIII2a}} \right \rrbracket
	+ \left \llbracket \raisebox{-8pt}{\includegraphics[height=.3in]{RIII2b}} \right \rrbracket
	- \left \llbracket \raisebox{-8pt}{\includegraphics[height=.3in]{RIII2c}} \right \rrbracket
	- \left \llbracket \raisebox{-8pt}{\includegraphics[height=.3in]{RIII2d}} \right \rrbracket
	- [2n-4] \left \llbracket \raisebox{-8pt}{\includegraphics[height=.3in]{RIII2e}} \right \rrbracket .
	\end{eqnarray*}
\end{proposition}

\begin{proof} To prove the statement, one can use the same approach as in the previous propositions, namely evaluating $\left \llbracket \raisebox{-8pt}{\includegraphics[height=.3in]{RIII1a}} \right \rrbracket$  and $\left \llbracket \raisebox{-8pt}{\includegraphics[height=.3in]{RIII2a}} \right \rrbracket$ by summing over all bracket evaluations for all the associated oriented diagrams. To avoid cumbersome computations,  we use instead the fact that $\db{\,\,\cdot\, \,}$ is invariant under the Reidemester III move. That is, $\left \llbracket \raisebox{-8pt}{\includegraphics[height=.3in]{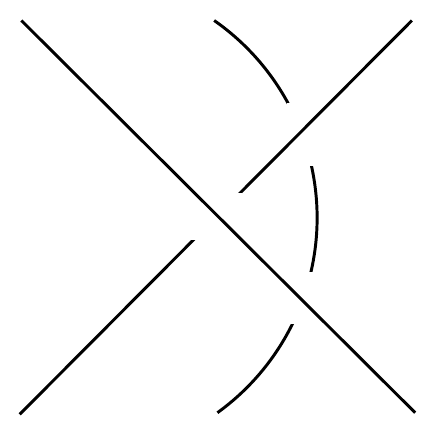}} \right \rrbracket = \left \llbracket \raisebox{-8pt}{\includegraphics[height=.3in]{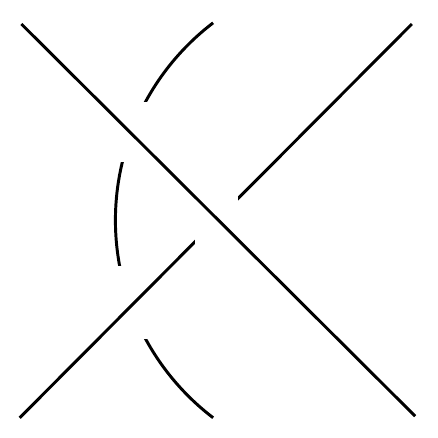}} \right \rrbracket $.
Using the skein relation~\eqref{eqn:skein-def}, we have 
\begin{eqnarray*}
\left \llbracket \raisebox{-8pt}{\includegraphics[height=.3in]{Rth1}} \right \rrbracket
	&=& q\left \llbracket \raisebox{-8pt}{\includegraphics[height=.3in]{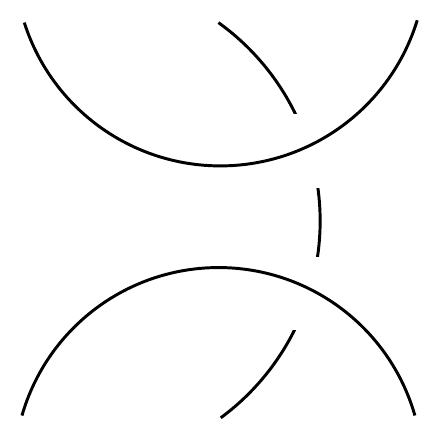}} \right \rrbracket
		+ q^{-1} \left \llbracket \raisebox{-8pt}{\includegraphics[height=.3in]{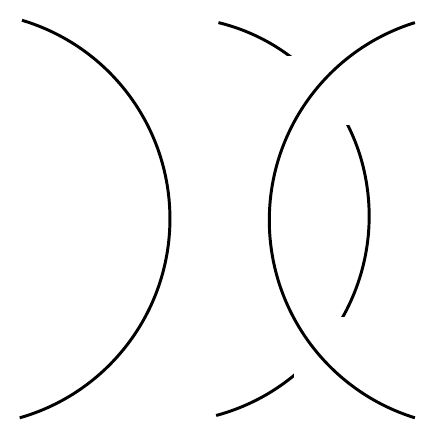}} \right \rrbracket
		- \left \llbracket \raisebox{-8pt}{\includegraphics[height=.3in]{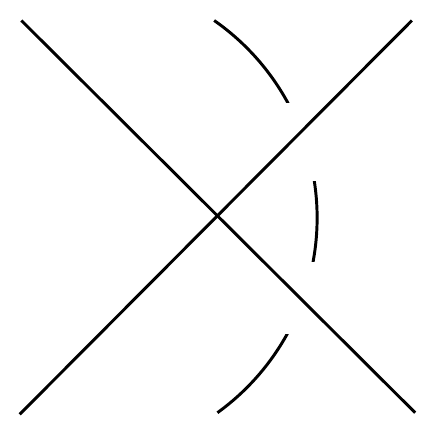}} \right \rrbracket \\
\left \llbracket \raisebox{-8pt}{\includegraphics[height=.3in]{Rth2}} \right \rrbracket
	&=& q\left \llbracket \raisebox{13pt}{\includegraphics[height=.3in, angle = 180]{Rth1a}} \right \rrbracket
		+ q^{-1} \left \llbracket \raisebox{-8pt}{\includegraphics[height=.3in]{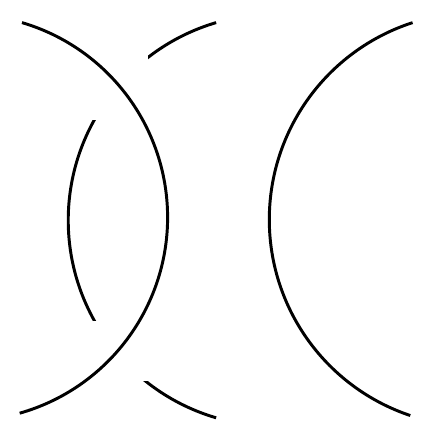}} \right \rrbracket
		- \left \llbracket \raisebox{-8pt}{\includegraphics[height=.3in]{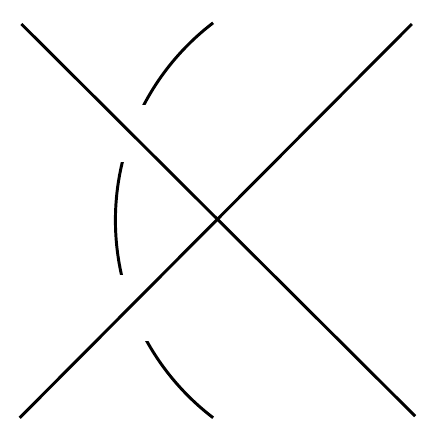}} \right \rrbracket .
\end{eqnarray*}
Since $\db{\,\,\cdot \,\,}$ is invariant under the Reidemeister II move, $\left \llbracket \raisebox{-8pt}{\includegraphics[height=.3in]{Rth1b}} \right \rrbracket = \left \llbracket \raisebox{-8pt}{\includegraphics[height=.3in]{Rth2b}} \right \rrbracket$, and we obtain that
$\left \llbracket \raisebox{-8pt}{\includegraphics[height=.3in]{Rth2c}} \right \rrbracket 
= \left \llbracket \raisebox{-8pt}{\includegraphics[height=.3in]{Rth1c}} \right \rrbracket$. Using again the skein relation~\eqref{eqn:skein-def}, we have 
\begin{eqnarray*}
0 &=& \left \llbracket \raisebox{-8pt}{\includegraphics[height=.3in]{Rth2c}} \right \rrbracket
		- \left \llbracket \raisebox{-8pt}{\includegraphics[height=.3in]{Rth1c}} \right \rrbracket  \\
	&=& q\left \llbracket \raisebox{-8pt}{\includegraphics[height=.3in]{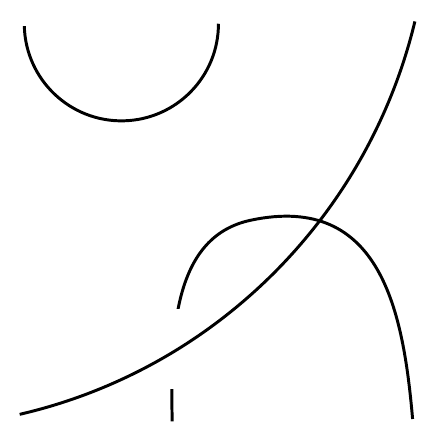}} \right \rrbracket
		+q^{-1}\left \llbracket \raisebox{-8pt}{\includegraphics[height=.3in]{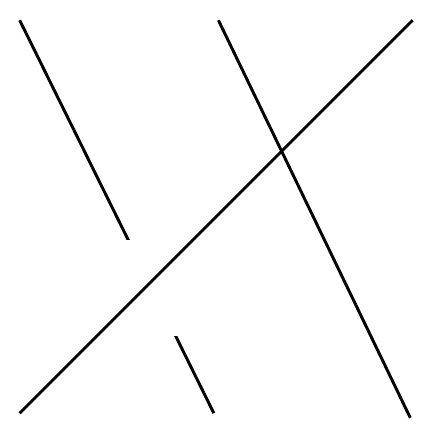}} \right \rrbracket
		-\left \llbracket \raisebox{-8pt}{\includegraphics[height=.3in]{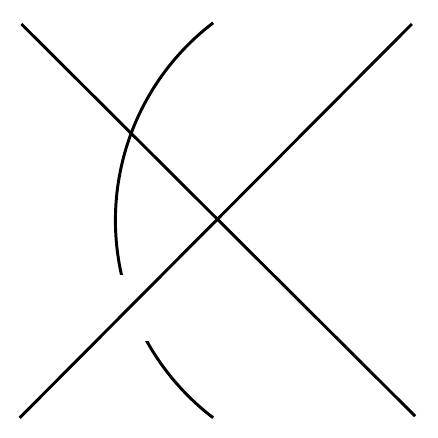}} \right \rrbracket 
		- \left (q\left \llbracket \raisebox{-8pt}{\includegraphics[height=.3in]{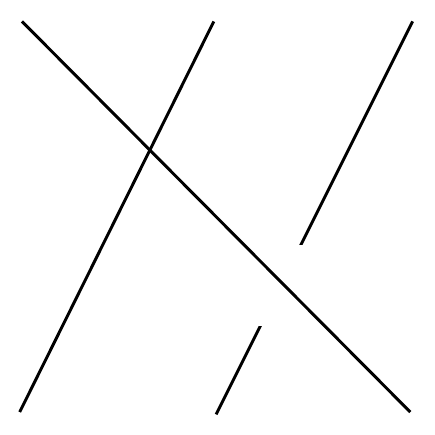}} \right \rrbracket
		+ q^{-1} \left \llbracket \raisebox{-8pt}{\includegraphics[height=.3in]{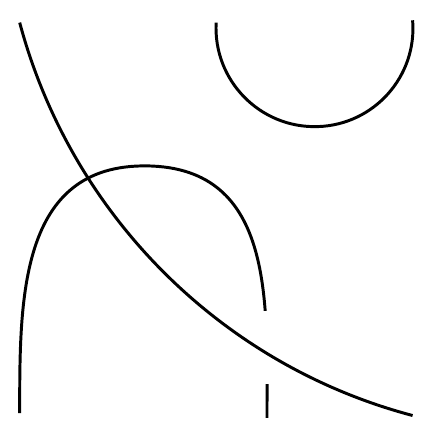}} \right \rrbracket
		- \left \llbracket \raisebox{-8pt}{\includegraphics[height=.3in]{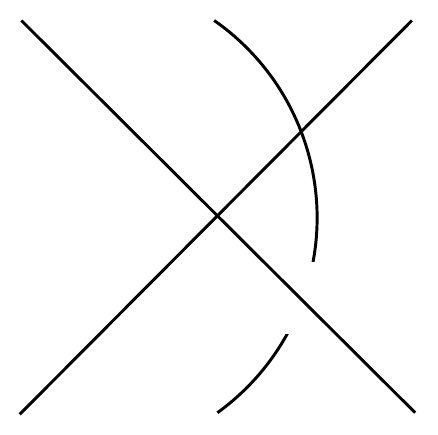}} \right \rrbracket \right ) \\
	&=& q \left ( q \left \llbracket \raisebox{-8pt}{\includegraphics[height=.3in]{RIII2c}} \right \rrbracket
		+ q^{-1} \left \llbracket \raisebox{-8pt}{\includegraphics[height=.3in]{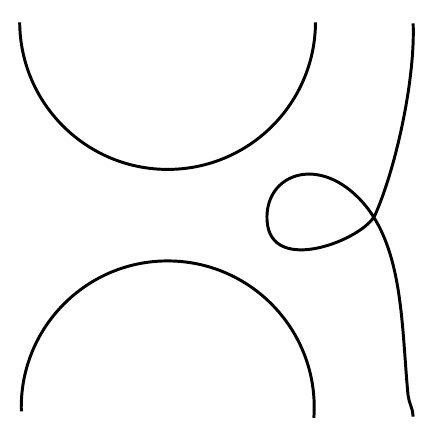}} \right \rrbracket
		- \left \llbracket \raisebox{-8pt}{\includegraphics[height=.3in]{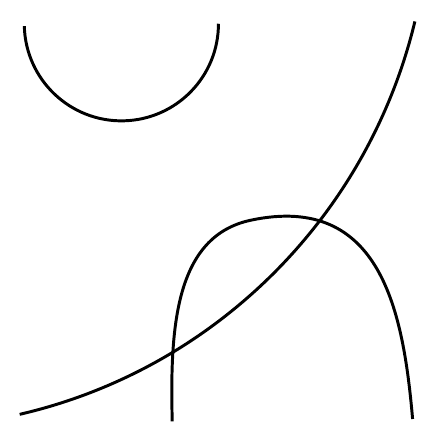}} \right \rrbracket \right ) 
		+ q^{-1} \left ( q \left \llbracket \raisebox{-8pt}{\includegraphics[height=.3in]{RIII2b}} \right \rrbracket
		+ q^{-1} \left \llbracket \raisebox{-8pt}{\includegraphics[height=.3in]{RIII2d}} \right \rrbracket
		- \left \llbracket \raisebox{-8pt}{\includegraphics[height=.3in]{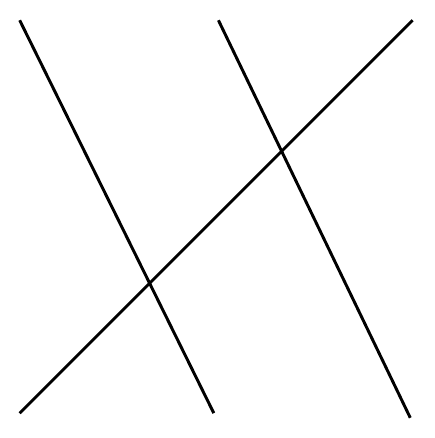}} \right \rrbracket \right ) \\
		&& - \left ( q\left \llbracket \raisebox{-8pt}{\includegraphics[height=.3in]{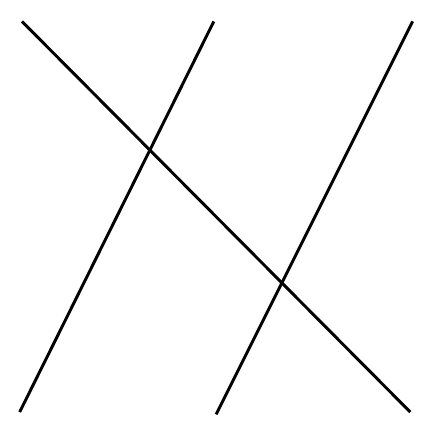}} \right \rrbracket
		+ q^{-1} \left \llbracket \raisebox{-8pt}{\includegraphics[height=.3in]{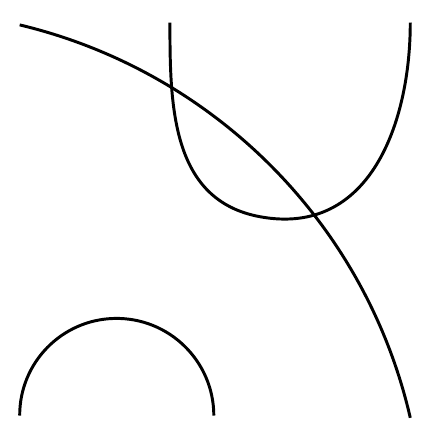}} \right \rrbracket
		- \left \llbracket \raisebox{-8pt}{\includegraphics[height=.3in]{RIII2a}} \right \rrbracket \right ) 
		- q\left ( q\left \llbracket \raisebox{-8pt}{\includegraphics[height=.3in]{RIII1d}} \right \rrbracket
		+ q^{-1} \left \llbracket \raisebox{-8pt}{\includegraphics[height=.3in]{RIII1b}} \right \rrbracket
		- \left \llbracket \raisebox{-8pt}{\includegraphics[height=.3in]{Rth2j}} \right \rrbracket \right ) \\
		&& -q^{-1} \left ( q\left \llbracket \raisebox{-8pt}{\includegraphics[height=.3in]{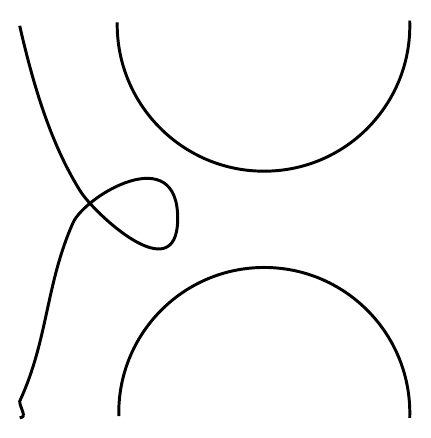}} \right \rrbracket
		+ q^{-1} \left \llbracket \raisebox{-8pt}{\includegraphics[height=.3in]{RIII1c}} \right \rrbracket
		- \left \llbracket \raisebox{-8pt}{\includegraphics[height=.3in]{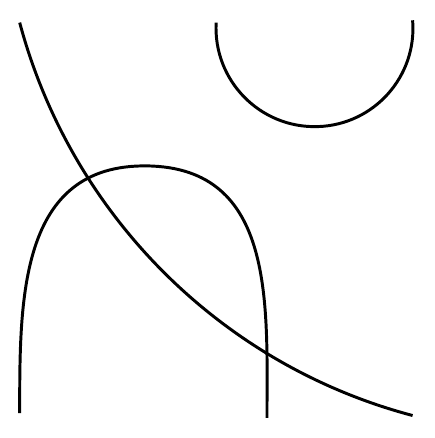}} \right \rrbracket \right ) 
		+ q\left \llbracket \raisebox{-8pt}{\includegraphics[height=.3in]{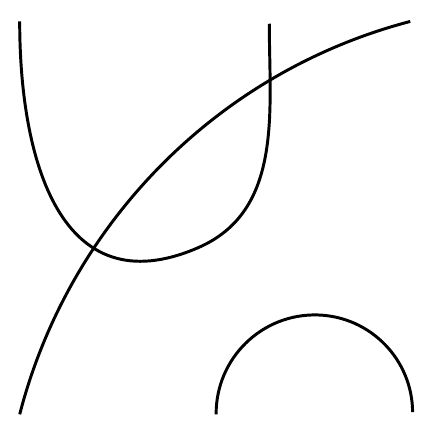}} \right \rrbracket
		+ q^{-1} \left \llbracket \raisebox{-8pt}{\includegraphics[height=.3in]{Rth2i}} \right \rrbracket
		- \left \llbracket \raisebox{-8pt}{\includegraphics[height=.3in]{RIII1a}} \right \rrbracket .
\end{eqnarray*}
Applying Proposition~\ref{R1} and canceling terms, we arrive to 
\begin{eqnarray*}
0 &=& \left \llbracket \raisebox{-8pt}{\includegraphics[height=.3in]{Rth2c}} \right \rrbracket
		- \left \llbracket \raisebox{-8pt}{\includegraphics[height=.3in]{Rth1c}} \right \rrbracket  \\
	&=& q^2 \left \llbracket \raisebox{-8pt}{\includegraphics[height=.3in]{RIII2c}} \right \rrbracket
		+ ([2n-2]+[2])\left \llbracket \raisebox{-8pt}{\includegraphics[height=.3in]{RIII2e}} \right \rrbracket
		- q\left \llbracket \raisebox{-8pt}{\includegraphics[height=.3in]{Rth2h}} \right \rrbracket
		+ \left \llbracket \raisebox{-8pt}{\includegraphics[height=.3in]{RIII2b}} \right \rrbracket
		+ q^{-2} \left \llbracket \raisebox{-8pt}{\includegraphics[height=.3in]{RIII2d}} \right \rrbracket \\
		&&- q^{-1} \left \llbracket \raisebox{-8pt}{\includegraphics[height=.3in]{Rth2k}} \right \rrbracket
		+ \left \llbracket \raisebox{-8pt}{\includegraphics[height=.3in]{RIII2a}} \right \rrbracket
		- q^2 \left \llbracket \raisebox{-8pt}{\includegraphics[height=.3in]{RIII1d}} \right \rrbracket 
		- \left \llbracket \raisebox{-8pt}{\includegraphics[height=.3in]{RIII1b}} \right \rrbracket 
		- ([2n-2]+[2]) \left \llbracket \raisebox{-8pt}{\includegraphics[height=.3in]{RIII1e}} \right \rrbracket \\
		&&- q^{-2} \left \llbracket \raisebox{-8pt}{\includegraphics[height=.3in]{RIII1c}} \right \rrbracket
		+ q^{-1} \left \llbracket \raisebox{-8pt}{\includegraphics[height=.3in]{Rth1i}} \right \rrbracket 
		+ q\left \llbracket \raisebox{-8pt}{\includegraphics[height=.3in]{Rth1j}} \right \rrbracket
		- \left \llbracket \raisebox{-8pt}{\includegraphics[height=.3in]{RIII1a}} \right \rrbracket .
\end{eqnarray*}
Now, from Proposition \ref{R2}, we have that \\
\[ \left \llbracket \raisebox{-8pt}{\includegraphics[height=.3in]{Rth2h}} \right \rrbracket
	= [2]\left \llbracket \raisebox{-8pt}{\includegraphics[height=.3in]{RIII2c}} \right \rrbracket
	+([2n-3]+1)\left \llbracket \raisebox{-8pt}{\includegraphics[height=.3in]{RIII2e}} \right \rrbracket, \\
\left \llbracket \raisebox{-8pt}{\includegraphics[height=.3in]{Rth2k}} \right \rrbracket
	=[2]\left \llbracket \raisebox{-8pt}{\includegraphics[height=.3in]{RIII2d}} \right \rrbracket
	+([2n-3]+1)\left \llbracket \raisebox{-8pt}{\includegraphics[height=.3in]{RIII2e}} \right \rrbracket, 
\]
\[
\left \llbracket \raisebox{-8pt}{\includegraphics[height=.3in]{Rth1i}} \right \rrbracket
	=[2]\left \llbracket \raisebox{-8pt}{\includegraphics[height=.3in]{RIII1c}} \right \rrbracket
	+([2n-3]+1)\left \llbracket \raisebox{-8pt}{\includegraphics[height=.3in]{RIII1e}} \right \rrbracket, \\
\left \llbracket \raisebox{-8pt}{\includegraphics[height=.3in]{Rth1j}} \right \rrbracket
	= [2]\left \llbracket \raisebox{-8pt}{\includegraphics[height=.3in]{RIII1d}} \right \rrbracket
	+([2n-3]+1) \left \llbracket \raisebox{-8pt}{\includegraphics[height=.3in]{RIII1e}} \right \rrbracket.
\]
Making the above replacements and combining like terms gives us
\begin{eqnarray*}
0 &=& (q^2-q[2])\left \llbracket \raisebox{-8pt}{\includegraphics[height=.3in]{RIII2c}} \right \rrbracket
	+ ([2n-2]+[2]-q[2n-3]-q-q^{-1}[2n-3]-q^{-1})\left \llbracket \raisebox{-8pt}{\includegraphics[height=.3in]{RIII2e}} \right
		 \rrbracket \\
	&&+ \left \llbracket \raisebox{-8pt}{\includegraphics[height=.3in]{RIII2b}} \right \rrbracket
	+ (q^{-2}-q^{-1}[2])\left \llbracket \raisebox{-8pt}{\includegraphics[height=.3in]{RIII2d}} \right \rrbracket
	+ \left \llbracket \raisebox{-8pt}{\includegraphics[height=.3in]{RIII2a}} \right \rrbracket 
	+ (q[2]-q^2)\left \llbracket \raisebox{-8pt}{\includegraphics[height=.3in]{RIII1d}} \right \rrbracket \\
	&&- \left \llbracket \raisebox{-8pt}{\includegraphics[height=.3in]{RIII1b}} \right \rrbracket
	+(q^{-1}[2n-3]+q^{-1}+q[2n-3]+q-[2n-2]-[2])\left \llbracket \raisebox{-8pt}{\includegraphics[height=.3in]{RIII1e}} \right
		 \rrbracket \\
	&&+(q^{-1}[2]-q^{-2})\left \llbracket \raisebox{-8pt}{\includegraphics[height=.3in]{RIII1c}} \right \rrbracket
	-\left \llbracket \raisebox{-8pt}{\includegraphics[height=.3in]{RIII1a}} \right \rrbracket \\
&=& \left \llbracket \raisebox{-8pt}{\includegraphics[height=.3in]{RIII2a}} \right \rrbracket
	+ \left \llbracket \raisebox{-8pt}{\includegraphics[height=.3in]{RIII2b}} \right \rrbracket
	- \left \llbracket \raisebox{-8pt}{\includegraphics[height=.3in]{RIII2c}} \right \rrbracket
	- \left \llbracket \raisebox{-8pt}{\includegraphics[height=.3in]{RIII2d}} \right \rrbracket
	- [2n-4] \left \llbracket \raisebox{-8pt}{\includegraphics[height=.3in]{RIII2e}} \right \rrbracket \\
	&& - \left (\left \llbracket \raisebox{-8pt}{\includegraphics[height=.3in]{RIII1a}} \right \rrbracket
	+\left \llbracket \raisebox{-8pt}{\includegraphics[height=.3in]{RIII1b}} \right \rrbracket
	-\left \llbracket \raisebox{-8pt}{\includegraphics[height=.3in]{RIII1c}} \right \rrbracket
	-\left \llbracket \raisebox{-8pt}{\includegraphics[height=.3in]{RIII1d}} \right \rrbracket
	-[2n-4]\left \llbracket \raisebox{-8pt}{\includegraphics[height=.3in]{RIII1e}} \right \rrbracket \right ),
\end{eqnarray*}
and the statement follows.
\end{proof}

Propositions~\ref{R0} through \ref{R3} provide consistent and sufficient skein relations to evaluate any planar unoriented 4-valent graph. In addition, the skein relation~\eqref{eqn:skein-def} together with these propositions yield a state summation model for the $SO(2n)$ Kauffman polynomial. 

\textbf{Acknowledgements.} The first author would like to thank Lorenzo Traldi for his useful comment and question via e-mail after the paper~\cite{CT} appeared on arxiv, which motivated this work. The second author was partially supported by an Undergraduate Research Grant from the California State University, Fresno to participate in the research of this paper.



\begin{thebibliography}{999}

\bibitem{CT} C. Caprau, J. Tipton, {\em The Kauffman polynomial and trivalent graphs}, will appear in Kyungpook Mathematical Journal; arXiv:math.GT/1107.1210.

\bibitem{C} R.P. Carpentier, {\em From planar graphs to embedded graphs - a new approach to Kauffman and Vogel's polynomial}, J. Knot Theory Ramifications \textbf{9}, Issue 8 (2000), 975-986.

\bibitem{HOMFLY} P. Freyd, D. Yetter, J. Hoste, W.B.R. Lickorish, K. Millett and A. Ocneanu,  {\em A new polynomial invariant of knots and links}, Bull. Am. Math. Soc. \textbf{12} (1985), 239-246.

\bibitem{K1} L.H. Kauffman, {\em An invariant of regular isotopy}, Trans. Amer. Math. Soc. \textbf{318} No. 2 (1990), 417-471.

\bibitem{K2} L.H. Kauffman, \textbf{Knots and Physics}, Third edition. Series on Knots and Everything, Vol. 1, World
Sci. Pub. (2001).

\bibitem{KV} L.H. Kauffman, P. Vogel, {\em Link polynomials and a graphical calculus}, J. Knot Theory Ramifications \textbf{1} (1992), 59-104.

\bibitem{MOY} H. Murakami, T. Ohtsuki, S. Yamada, {\em Homfly polynomial via an invariant of colored plane graphs}, L'Enseignement Mathematique, \textbf{44} (1998), 325-360.

\bibitem{PT} J.H. Przytycki, P. Traczyk, {\em Invariants of links of Conway type}, Kobe J. Math. \textbf{2} (1987), 115-139.

\bibitem{WU} H. Wu, {\em On the Kauffman-Vogel and the Murakami-Ohtsuki-Yamada graph polynomials}, {\em J. Knot Theory Ramifications} \textbf{21}, No. 10 (2012) 1250098 (40 pages).


\end{thebibliography}
\end{document}